\documentclass[reqno,12pt]{amsart}
\usepackage{amscd,amsfonts,mathrsfs,amsthm,amssymb, amsmath}
\usepackage{enumerate}
\usepackage{multicol}        			
\usepackage{color}           			
\usepackage{array}
\usepackage{ae}
\usepackage{accents}
\usepackage{epsfig}
\usepackage[e]{esvect}
\usepackage[alphabetic, sorted-cites]{amsrefs}
\usepackage[normalem]{ulem}
\usepackage[makeroom]{cancel}

\newcommand{\n}{\mathbb{N}}

\newcommand{\ip}{(-\pi/2,\pi/2)} 
\newcommand{\gp}{\mathscr{G}} 

\def\comment#1 {{\color{red}(Comment: #1) }}

\def\H{\bf {H}}

\def\v{\operatorname{v}}
\def\dist      {\operatorname{dist}}

\def\real     #1{{\mathbb R^{#1}}}

\def\natural  #1{{\mathbb N^{#1}}}

\def\gind {\operatorname{g}}

\newtheorem{theorem}{Theorem}[section]

\newtheorem{lemma}[theorem]{Lemma}
\newtheorem{claim}{Claim}
\newtheorem*{thma}{Theorem}

\newtheorem{definition}[theorem]{Definition}
\theoremstyle{definition}
\newtheorem{remark}[theorem]{Remark}

\def\pproof#1{\@ifnextchar[\opargproof
{\opargproof[\it Proof of #1.]}}
\def\opargproof[#1]{\par\noindent {\bf #1 }}

\makeatother

\numberwithin{equation}{section}

\begin{document}

\title[Translating solitons]{A characterization of\\ the grim reaper cylinder}

\author[F. Mart\'in]{\textsc{F. Mart\'in}}
\author[J. P\'erez-Garc\'ia]{\textsc{J. P\'erez-Garc\'ia}}
\author[A. Savas-Halilaj]{\textsc{A. Savas-Halilaj}}
\author[K. Smoczyk]{\textsc{K. Smoczyk}}

\address{Francisco Mart\'in \& Jes\'us P\'erez-Garc\'ia\newline
Departamento de Geometr\'ia y Topolog\'ia\newline
Instituto Espa\~{n}ol de Matem\'aticas IEMath-GR\newline
Universidad de Granada\newline
18071 Granada, Spain\newline
{\sl e-mail address of F. Mart\'in:} {\bf fmartin@ugr.es}\newline
{\sl e-mail address of J. P\'erez-Garc\'ia:} {\bf jpgarcia@ugr.es}
}

\address{Andreas Savas-Halilaj \& Knut Smoczyk\newline
Institut f\"ur Differentialgeometrie and \newline
Riemann Center for Geometry and Physics\newline
Leibniz Universit\"at Hannover\newline
Welfengarten 1\newline
30167 Hannover, Germany\newline
{\sl e-mail address of A. Savas-Halilaj:} {\bf savasha@math.uni-hannover.de}\newline
{\sl e-mail address of K. Smoczyk:} {\bf smoczyk@math.uni-hannover.de}
}

\date{}
\thanks{F. Mart\'in and J. P\'erez-Garc\'ia are partially supported by MINECO-FEDER
grant no. MTM2014-52368. J. P\'erez-Garc\'ia is also supported by MINECO (FPI grant, BES-2012-055302)
and A. Savas-Halilaj \& K. Smoczyk by DFG SM 78/6-1.}

\begin{abstract}
In this article we prove that a connected and properly embedded translating soliton in $\real{3}$ with uniformly bounded genus on compact sets
which is $C^1$-asymptotic to two planes outside a cylinder, either is flat or coincides with the grim reaper cylinder.
\end{abstract}

\maketitle

\section{Introduction}

An oriented smooth surface \(f:M^2\to\real{3}\) is called {\it translating soliton} of the mean curvature flow (\emph{translator} for short)
if its mean curvature vector field \(\H\) satisfies the differential equation
\begin{equation*}
\H=\v^\perp,
\end{equation*}
where \(\v\in\real{3}\) is a fixed vector of unit length and \(\v^\perp\) stands for the orthogonal projection of \(\v\) to the normal bundle of the immersion \(f\).
If \(\xi\) is the outer unit normal of \(f\), then the translating property can be expressed in terms of scalar quantities as
\begin{equation}\label{eq:translating_equation_mcf}
H:=-\langle \H, \xi \rangle=-\langle \v, \xi \rangle,
\end{equation}
where \(H\) is the scalar mean curvature of \(f\).
Translators are important in the singularity theory of the mean curvature flow since they often occur as Type-II singularities.
An interesting example of a translator is the {\it canonical grim reaper cylinder} 
$\gp$ which can be represented parametrically via the embedding $u:\ip\times\real{}\to\real{3}$ given by
$$u(x_1,x_2)=(x_1,x_2,-\log\cos x_1).$$
Any translator in the direction of $\v$ which is an euclidean product of a planar
curve and $\real{}$ is either a plane containing $\v$ or
can be obtained by a suitable combination of a rotation and a dilation of the canonical
grim reaper cylinder. The latter examples will be called {\it grim reaper cylinders}.
Note that the canonical grim reaper cylinder $\gp$ is translating with respect to the direction $\v=(0,0,1).$
For simplicity we will assume that all translators to be considered here are translating in the direction
$\v=(0,0,1)$.

Before stating the main theorem let us set up the notation and provide some definitions.

\begin{definition}\label{asymp-1}
Let $\mathcal{H}$ be an open half-plane in $\real{3}$ and $\operatorname{w}$ the unit inward
pointing normal of $\partial\mathcal{H}$. For a fixed positive number $\delta$, denote by
$\mathcal{H}_{\delta}$ the set given by
$$\mathcal{H}_{\delta}:=\big\{p+t\operatorname{w}:p\in\partial\mathcal{H} \quad\text{and}\quad t>\delta\big\}.$$
\begin{enumerate}
\item [\rm (a)]
We say that a smooth surface $M$ is $C^k$-asymptotic to the open half-plane $\mathcal{H}$ if $M$
can be represented as the graph of a $C^k$-function $\varphi:\mathcal{H}\to\real{}$ such that for every
$\varepsilon>0$ there exists $\delta>0$ so that for any $j\in\{1,2,\dots,k\}$ it holds
\begin{equation*}
{\sup}_{\mathcal{H}_{\delta}}|\varphi|<\varepsilon\quad\text{and}\quad
{\sup}_{\mathcal{H}_{\delta}}|D^j\varphi|<\varepsilon.
\end{equation*}
\item [\rm (b)]
A smooth surface $M$ is called $C^k$-asymptotic outside a cylinder
to two half-planes $\mathcal{H}_1$ and $\mathcal{H}_2$ if there exists a
solid cylinder $\mathcal{C}$ such that:
\smallskip
\begin{enumerate}
\item[\rm ($b_1$)]
the solid cylinder $\mathcal{C}$ contains the boundaries of the half-planes $\mathcal{H}_1$ and $\mathcal{H}_2$,
\smallskip
\item[\rm ($b_2$)]
the set $M-\mathcal{C}$ consists of two connected components $M_1$ and $M_2$
that are $C^1$-asymptotic to $\mathcal{H}_1$ and $\mathcal{H}_2$, respectively.
\end{enumerate}
\end{enumerate}
\end{definition}

\begin{figure}[h]
\includegraphics[scale=.07]{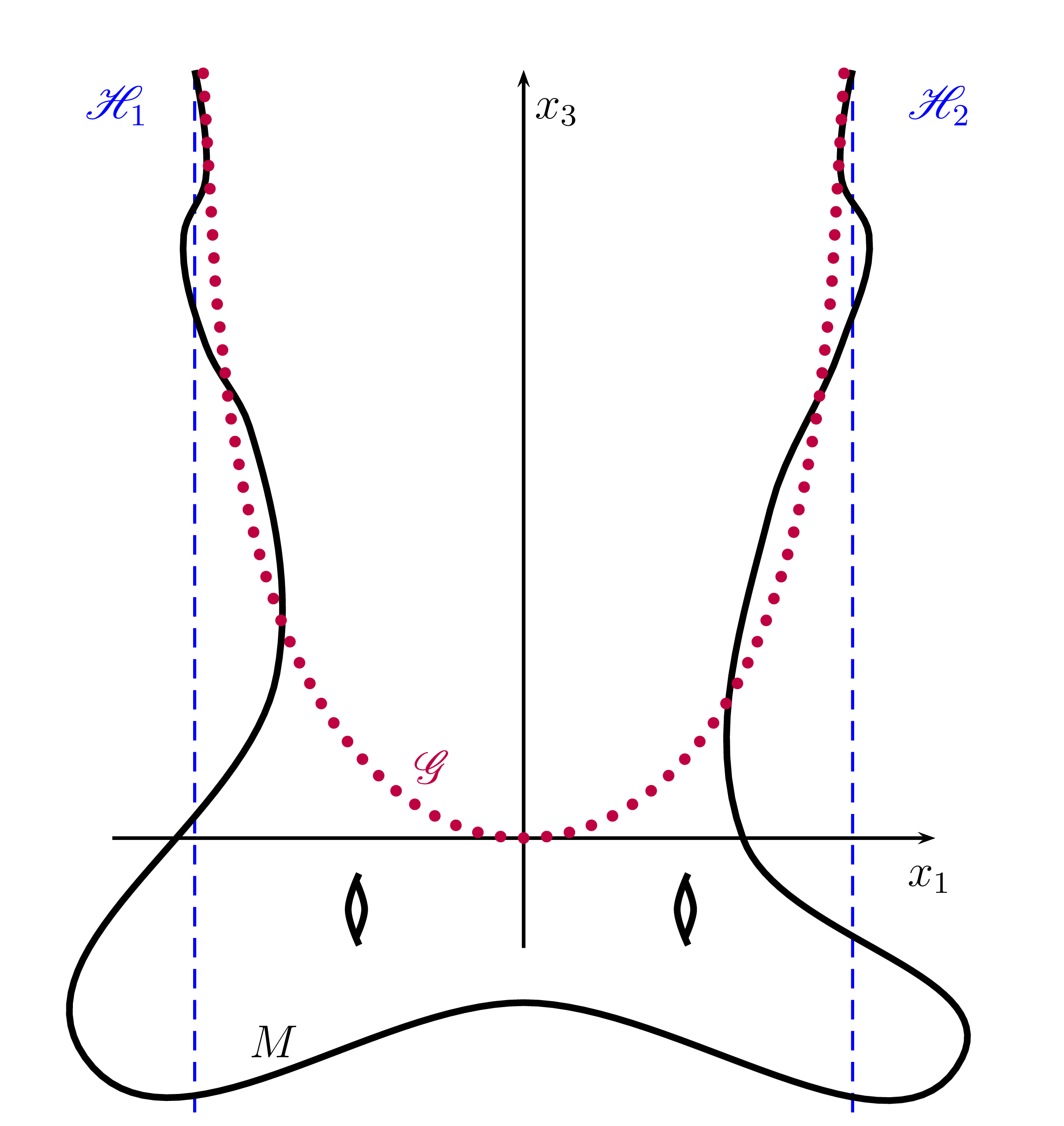}\caption{Asymptotic behavior}\label{asbehavior}
\end{figure}

For example the canonical grim reaper cylinder $\gp$ is asymptotic to the parallel half-planes
$$\mathcal{H}_{1}=\big\{(x_1,x_2,x_3)\in\real{3}:x_3>r_0>0,\,x_1=-{\pi}/{2}\big\}$$
and
$$\mathcal{H}_{2}=\big\{(x_1,x_2,x_3)\in\real{3}:x_3>r_0>0,\,x_1=+{\pi}/{2}\big\}$$
outside the solid cylinder
$${C}=\big\{(x_1,x_2,x_3)\in\real{3}:x^2_1+x^2_3\le r^2_0+{\pi^2}/{4}\big\},$$
where here $r_0$ is a positive real constant.

Let us now state our main result.

\begin{thma}
Let \(f:M^2\to\real{3}\) be a connected, properly embedded\,\footnote{Here by embedded we only mean that $M$ has no self-intersections.}
translating soliton with uniformly bounded genus on compact sets of $\real{3}$ and $\mathcal{C}$ be a solid cylinder whose axis is
perpendicular to the direction of translation of $M:=f(M^2)$. Assume that $M$ is \(C^{1}\)-asymptotic outside the cylinder $\mathcal{C}$ to
two half-planes whose boundaries belongs on $\partial \mathcal{C}$. Then either
\begin{enumerate}
\item [\rm (a)] both half-planes are contained in the same vertical plane $\Pi$ and $M=\Pi$, or
\item [\rm (b)] the half-planes are included in different parallel planes and $M$ coincides with a grim reaper cylinder.
\end{enumerate}
\end{thma}

\begin{remark}
Let us make here some remarks concerning our main theorem.
\begin{enumerate}[\rm (a)]
\item
Notice that in the above theorem infinite genus a priori could be possible.
The assumption that $M$ has uniformly bounded genus on compact sets of $\real{3}$ means
that for any positive $r$ there exists $m(r)$ such that for any $p\in M$ it holds
$$\operatorname{genus}\big\{M\cap \mathbb{B}_r(p)\big\}\le m(r),$$
where $\mathbb{B}_r(p)$ is the ball of radius $r$ in $\real{3}$ centered at the point $p$.
Roughly speaking, the above condition says that as we approach infinity the
``size of the  holes\rq{}\rq{} of $M$ is not becoming arbitrary small and furthermore they 
are not getting arbitrary close to each other.
\smallskip
\item
We would like to mention here that Nguyen \cite{nguyen1,nguyen2,nguyen3} constructed
examples of complete embedded translating solitons in the euclidean space $\real{3}$ with infinite genus. Outside
a cylinder, these examples look like a family of parallel half-planes. This means that the hypothesis
about the number of half-planes is sharp. Very recently,
D\'avila, Del Pino \& Nguyen \cite{davila} and, independently, Smith \cite{smith} constructed 
examples of complete embedded translators with finite non-trivial topology. For an exposition
of examples of translators see also \cite[Subsection 2.2]{fra14}.
\smallskip
\item
Ilmanen constructed a one-parameter family of complete convex translators, defined on strips,
connecting the grim reaper cylinder with the bowl soliton \cite{whi02}. Note that the level sets of these translators are closed curves.
This means that our hypothesis of being asymptotic to two planes outside a cylinder is natural and cannot be removed.
\end{enumerate}
\end{remark}

Let us describe now the general idea and the steps of the proof. As already mentioned, we will assume
that $\v=(0,0,1)$. Without loss of generality we can choose  the $x_2$-axis as the axis of rotation of $\mathcal{C}$. First we show that the half-planes must be
parallel to each other, they should be also parallel to the translating direction and that both wings of $M$ outside the cylinder must point in the direction of $\v$.
Then, after a translation in the direction of the $x_1$-axis, if necessary, we prove that the asymptotic half-planes $\mathcal{H}_1$ and $\mathcal{H}_2$ are subsets of the parallel
planes
$$\Pi(-{\pi}/{2})=\big\{(x_1,x_2,x_3)\in\real{3}:x_1=-{\pi}/{2}\big\}$$
and
$$\Pi(+{\pi}/{2})=\big\{(x_1,x_2,x_3)\in\real{3}:x_1=+{\pi}/{2}\big\},$$
respectively, and that $M$ is contained
in the slab between the planes $\Pi(-{\pi}/{2})$ and $\Pi(+{\pi}/{2})$. To prove this claim we study the
$x_1$-coordinate function of $M$ in order to control its range. By the strong maximum principle we
conclude that the $x_1$-coordinate function cannot attain local maxima or minima. To prove that
${\sup}_{M}x_1=\pi/2=-{\inf}_{M}x_1$ we perform a ``blow-down" argument
based on a compactness theorem of White \cite{whi15} for sequences of properly embedded minimal surfaces in Riemannian
$3$-manifolds. The next step is to show that $M$
is a bi-graph over $\Pi(+{\pi}/{2})$ and that the plane
$$\Pi(0)=\big\{(x_1,x_2,x_3)\in\real{3}:x_1=0\big\}$$
is a plane of symmetry for $M$. To prove this claim we use Alexandrov's method of moving planes. In the sequel we show that $M$ must be a
graph over a slab of the $x_1x_2$-plane.  Thus, $M$ must have zero genus and it must be strictly mean convex. To achieve this goal we carefully investigate the set of the local
maxima and minima of the profile curve
$$\Gamma=M\cap\Pi(0)\subset\mathcal{C}.$$
Performing again a ``blow-down" argument along the ends of the curve $\Gamma$ we deduce that $M$ looks like a grim reaper cylinder
at infinity. To finish the proof, we consider the function $\xi_2$ which measures the $x_2$-coordinate
of the Gau{\ss} map $\xi$ of $M$. Then, by applying the strong maximum principle to $\xi_2H^{-1}$, we deduce that $\xi_2$ is identically zero.
This implies that the Gau{\ss} curvature of $M$ is zero and so $M$ must coincide with a grim reaper cylinder (see \cite[Theorem B]{fra14}).

The structure of the paper is as follows. In Section \ref{sec:main_tools} we introduce the
tangency principle, the compactness and the strong barrier principle of White \cite{whi12, whi15}.
In Section 3 we present a lemma that will play a crucial role in the proof of our theorem. 
This lemma (Lemma \ref{pro:limit_process}) asserts that every complete, properly embedded translating soliton in $\real{3}$ with
the asymptotic behavior of two half-planes has a surprising amount of internal
dynamical periodicity. The main theorem is proved in Section 
\ref{sec:proof_of_the_main_theorem}.

%
%
\section{A compactness theorem and a strong barrier principle}\label{sec:main_tools}
We  will introduce here the main tools that we will use in the proofs.

\subsection{The tangency principle}
According to this maximum principle
(see \cite[Theorem 2.1]{fra14}), two different translators 
cannot ``touch'' each other at one interior or boundary point. More precisely:

\begin{theorem}\label{thm:tangency_principle}
Let $\Sigma_1$ and $\Sigma_2$ be embedded connected translators in $\real{3}$ with boundaries $\partial \Sigma_1$ and $\partial \Sigma_2$.
\begin{enumerate}
\item [\rm (a)] $(${\bf Interior principle}$)$ Suppose that there exists a common point $x$ in the interior of $\Sigma_1$ and $\Sigma_2$ where the corresponding tangent
planes coincide and such that \(\Sigma_{1}\) lies at one side of \(\Sigma_{2}\). Then $\Sigma_1$ coincides with $\Sigma_2$.
\smallskip
\item [\rm (b)] $(${\bf Boundary principle}$)$
Suppose that the boundaries $\partial \Sigma_1$ and $\partial \Sigma_2$ lie in the same plane $\Pi$ and that the intersection of \(\Sigma_{1}\), \(\Sigma_{2}\) with \(\Pi\)
is transversal. Assume that \(\Sigma_{1}\) lies at one side of \(\Sigma_{2}\) and that there exists a common point of $\partial \Sigma_1$ and $\partial \Sigma_2$ where the
surfaces
$\Sigma_1$ and $\Sigma_2$
have the same tangent plane. Then $\Sigma_1$ coincides with $\Sigma_2$.
\end{enumerate}
\end{theorem}

\subsection{A compactness theorem for minimal surfaces}
Let $\Sigma$ be a surface in a $3$-manifold $(\Omega,g)$. Given $p\in\Sigma$ and
$r>0$ we denote by
$$D_r(p):=\big\{w\in T_p\Sigma:|w|<r\big\}$$
the tangent disc of radius $r$. Consider now $T_p\Sigma$ as a vector subspace of $T_p\Omega$
and let $\nu$ be the unit normal vector of $T_p\Sigma$ in $T_p\Omega$. Fix a sufficiently small $\varepsilon>0$ and denote by
$W_{r,\varepsilon}(p)$ the solid cylinder around $p$, that is
$$
W_{r,\varepsilon}(p):=\big\{\exp_p(q+t\nu_q):q\in D_r(p)\,\,\text{and}\,\,
|t|\le\varepsilon\big\},
$$
where $\exp$ stands for the exponential map of the ambient Riemannian $3$-manifold $(\Omega,g)$.
Given a function $u:D_r(p)\to\real{}$, the set
$$\operatorname{Graph}(u):=\big\{\exp_p(q+u(q)\nu_q):q\in D_r(p)\big\}$$ 
is called the graph of $u$ over $D_r(p)$.
\begin{definition}[\bf Convergence in the $C^{\infty}$-topology]
Let $(\Omega,g)$ be a Riemannian $3$-manifold and $\{M_i\}_{i\in\natural{}}$
a sequence of connected embedded surfaces. The sequence $\{M_i\}_{i\in\natural{}}$ converges
in the $C^{\infty}$-topology with finite multiplicity to a smooth embedded surface $M_{\infty}$ if:
\begin{enumerate}
\item[\rm (a)] $M_{\infty}$ consists of accumulation points of $\{M_i\}_{i\in\natural{}}$, that is for each $p\in M_{\infty}$
there exists a sequence of points $\{p_i\}_{i\in\natural{}}$
such that $p_i\in M_i$, for each $i\in\natural{}$, and $p={\lim}_{i\to\infty} p_i$.
\smallskip
\item [\rm (b)]
For all $p\in M_{\infty}$ there exist $r,\varepsilon>0$ such that
$M_{\infty}\cap W_{r,\varepsilon}(p)$ can be represented as the graph of a function
$u$ over $D_r(p)$.
\smallskip
\item [\rm (c)]
For all large $i\in\natural{}$, the set
$M_i\cap W_{r,\varepsilon}(p)$ consists of a finite number $k$, independent of $i$,
of graphs of functions $u^1_i,\dots,u^k_i$ over $D_r(p)$ which converge smoothly to $u$.
\end{enumerate}
The multiplicity of a given point $p\in M_{\infty}$ is defined to be the number of graphs in
$M_i\cap W_{r,\varepsilon}(p)$, for $i$ large enough.
\end{definition}
\begin{remark}
Note that although each surface of the sequence $\{M_i\}_{i\in\natural{}}$ is connected
the limiting surface $M_{\infty}$ is not necessarily connected. However, the multiplicity
remains constant on each connected component $\Sigma$ of $M_{\infty}$.
For more details we refer to \cite{perez,choi}.
\end{remark}

\begin{definition}
Let $\{M_i\}_{i\in\natural{}}$ be a sequence of embedded surfaces in a Riemannian $3$-manifold
$(\Omega,g)$.
\begin{enumerate}
\item[\rm(a)]
We say that $\{M_i\}_{i\in\natural{}}$ has uniformly 
bounded area on compact subsets of $\Omega$ if
$${\limsup}_{i\to\infty}\operatorname{area}\{M_i\cap K\}<\infty,$$
for any compact subset $K$ of $\Omega$.
\medskip
\item[\rm(b)]
We say that
$\{M_i\}_{i\in\natural{}}$ has uniformly bounded genus on compact subsets of $\Omega$ if
$${\limsup}_{i\to\infty}\operatorname{genus}\big\{M_i\cap K\big\}<\infty,$$
for any compact subset $K$ of $\Omega$.
\end{enumerate}
\end{definition}

\begin{theorem}[\bf White\rq{s} compactness 
theorem]\label{thm:compactness_theorem_Brian_White}
Let $(\Omega,g)$ be an arbitrary Riemannian $3$-manifold.  Suppose that 
$\{M_{i}\}_{i\in\natural{}}$
is a sequence of connected properly embedded minimal surfaces. Assume that the area and the genus of
$\{M_{i}\}_{i\in\natural{}}$ are uniformly bounded on compact subsets of $\Omega$. Then,
after  passing to a subsequence, $\{M_{i}\}_{i\in\natural{}}$ converges to a
smooth properly embedded minimal surface $M_{\infty}\subset\Omega$. The convergence is smooth
away from a discrete set denoted by $\operatorname{Sing}$. Moreover, for each connected
component $\Sigma$ of $M_{\infty}$, either
\begin{enumerate}
\item [\rm (a)] the convergence to $\Sigma$ is smooth everywhere with multiplicity $1$, or
\smallskip
\item [\rm (b)] the convergence is smooth, with some multiplicity greater than one, away
from $\Sigma\cap\operatorname{Sing}$.
\end{enumerate}
Now suppose that $\Omega$ is an open subset of $\real{3}$ while the metric $g$ is not
necessarily flat. If $p_i=(p_{1i},p_{2i},p_{3i})\in M_i$, $i\in\natural{}$, converges to $p\in M_{\infty}$ then, after
passing to a further subsequence, either
$T_{p_i}M_i\to T_pM $
or there exists a sequence of real number $\{\lambda_i\}_{i\in\natural{}}$ tending to $\infty$
such that the sequence of surfaces $\{\lambda_i(M_i-p_i)\}_{i\in\natural{}}$, where
$$\lambda_i(M_i-p_i)=\big\{\lambda_i(x_1-p_{1i},x_2-p_{2i},x_3-p_{3i})
\in\real{3}:(x_1,x_2,x_3)\in M\big\},$$
converge smoothly and with multiplicity $1$ to a non-flat, complete and properly embedded
minimal surface ${M}^*_{\infty}$ of finite total curvature and with ends parallel to
$T_pM_{\infty}$.
\end{theorem}

A crucial assumption in the compactness theorem of White is that the sequence has uniformly
bounded area on compact subsets of $\Omega$. Let us denote by
$$\mathscr{Z}:=\big\{p\in \Omega: {\limsup}_{i\to\infty}{\operatorname{area}}\{M_{i}\cap \mathbb{B}_r(p)\}
= \infty \text{ for every } r>0\big\},$$
the set where the area blows up. Clearly $\mathscr{Z}$ is a closed set. It will be useful to have conditions that will imply that the set $\mathscr{Z}$ is empty. 
In this direction, White \cite[Theorem 2.6 and Theorem 7.4]{whi12} shows that under some 
natural conditions the set $\mathscr{Z}$ satisfies 
the same maximum principle as properly embedded minimal surfaces without boundary.
\begin{theorem}[\bf White\rq{s} strong barrier principle]\label{thm:Controlling_area-blowup}
Let $(\Omega,g)$ be a Riemannian $3$-manifold and $\{M_{i}\}_{i\in\natural{}}$ a sequence of properly
embedded minimal surfaces,
with boundaries $\{\partial M_i\}_{i\in\natural{}}$ in $(\Omega,g)$. Suppose that:
\begin{enumerate}
\item[\rm (a)]
The lengths of $\{\partial M_i\}_{i\in\natural{}}$ are uniformly bounded on compact subsets of $\Omega$,
that is
$${\limsup}_{i\to\infty} {\operatorname{length}}\{\partial M_{i}\cap K\}< \infty,$$
for any relatively compact subset $K$ of $\Omega$.
\smallskip
\item[\rm (b)]
The set $\mathscr{Z}$ of $\{M_{i}\}_{i\in\natural{}}$ is contained in a closed
region $N$ of $\Omega$ with smooth, connected boundary $\partial N$ such that
$g\big(H_{\partial N},\xi\big)\ge 0,$
at every point of $\partial N$, where $H_{\partial N}(p)$ is the mean curvature vector of
$\partial N$ at $p$ and $\xi(p)$ is the unit normal at $p$ to the surface $\partial N$ that points into 
$N$.
\end{enumerate}
If the set $\mathscr{Z}$ contains
any point of $\partial N$, then it contains all of $\partial N$.
\end{theorem}

\begin{remark}
The above theorem is a sub-case of a more general result of White. In fact the strong barrier
principle of White holds for sequences of embedded hypersurfaces of $n$-dimensional Riemannian
manifolds which are not necessarily minimal but they have uniformly bounded mean curvatures. For
more details we refer to \cite{whi12}.
\end{remark}

\subsection{Distance in Ilmanen's metric}

Due to a result of Ilmanen \cite{ilmanen} there is a duality between translators in the euclidean space $\real{3}$ and  
minimal surfaces in $(\real{3},\operatorname{g})$, where $\operatorname{g}$ is the conformally flat Riemannian
metric
$$\operatorname{g}(\cdot\,,\cdot):=e^{x_3}\langle\cdot\,,\cdot\rangle,$$
and $\langle\cdot\,,\cdot\rangle$ stands for the euclidean inner product of 
$\real{3}$. The metric $\operatorname{g}$ will be called Ilmanen's metric. In particular, every translator in the euclidean space $\real{3}$ is a minimal surface in $(\real{3},\operatorname{g})$
and vice-versa. The Levi-Civita connection $D^{\operatorname{g}}$ of $\operatorname{g}$ is related to the Levi-Civita connection $D$
of the euclidean space via the relation
$$D^{\operatorname{g}}_{X}Y=D_XY+\frac{1}{2}\big\{\langle X,\partial_{x_3}\rangle Y+\langle Y,\partial_{x_3}\rangle X-\langle X,Y\rangle\partial_{x_3}\big\}.$$
One can check that parallel transports and rotations with respect to the euclidean metric that preserve $\v$ preserve the geodesics of $(\real{3},\operatorname{g})$.
Moreover, one can easily verify that vertical straight lines and ``grim-reaper-type" curves, i.e.,
images of smooth curves $\gamma:(-\pi,\pi)\to(\real{3},\operatorname{g})$ of the form
$$\gamma(t)=\big(t,0,-2\log\cos\tfrac{t}{2}\big),$$
are geodesics with respect to the Ilmanen's metric. Using the above mentioned transformations we can construct all the geodesics
of $(\real{3},\operatorname{g})$.
Let now $\delta$ be a sufficiently small positive number and $p=(p_1,p_2,p_3)$ a point in $\real{3}$ such that $p_1\in(-\delta,0)$ and $p_3>0$. Let us
denote by ${\dist}_{\operatorname{g}}(p,\Pi(0))$ the distance of $p$ from the plane
$$\Pi(0)=\big\{(x_1,x_2,x_3)\in\real{3}:x_1=0\big\}.$$
with respect to the Ilmanen's metric and by $\dist(p,\Pi(0))=-p_1$ the euclidean distance of the point $p$ from
the plane $\Pi(0)$. The distance ${\dist}_{\operatorname{g}}(p,\Pi(0))$ is given as the length with respect to the Ilmanen's metric
of the smooth curve $l:(p_1,0)\to(\real{3},\operatorname{g})$ given by
$$l(t)=\big(t,p_2,-2\log\cos \tfrac{t}{2}+2\log\cos \tfrac{p_1}{2}+p_3\big)$$
A direct computation shows that
\begin{equation*}
{\dist}_{\operatorname{g}}(p,\Pi(0))=\int^0_{p_1}e^{\tfrac{p_3}{2}}\,\cdot\frac{\cos\tfrac{p_1}{2}}{\cos\tfrac{t}{2}}\cdot
\sqrt{1+\big(\tan\tfrac{t}{2}\big)^2}dt=2e^{\tfrac{p_3}{2}}\cdot\sin\tfrac{{\dist}(p,\Pi(0))}{2}.
\end{equation*}
From the above formula we immediately obtain the following result which will be very useful in the last step
of the proof of our theorem.
\begin{lemma}\label{asymptotic-lemma}
Suppose that $M$, regarded as a minimal surface in $(\real{3},\operatorname{g})$, is $C^{\infty}$-asymptotic to two parallel
vertical half-planes
$\mathcal{H}_1$ and $\mathcal{H}_2$ outside the cylinder $\mathcal{C}$. Then the translator $M$
is also smoothly asymptotic to the above mentioned half-planes outside $\mathcal{C}$ with respect to the euclidean metric.
\end{lemma}

\section{A compactness result and its first consequences}
The translating property is preserved if we act on $M$ via isometries of $\real{3}$ which preserves
the translating direction. Therefore, if $(a,b,c)$ is a vector of $\real{3}$ then the surface
$$M+(a,b,c)=\big\{(x_1+a,x_2+b,x_3+c)\in\real{3}:(x_1,x_2,x_3)\in M\big\}$$
is again a translator.
Based on White's compactness theorem, we can prove a convergence result for some special sequences
of translating solitons. More precisely, we show the following:

\begin{lemma}\label{pro:limit_process}
Let $M$ be a surface as in our theorem.
Suppose that
$\{b_{i}\}_{i\in\natural{}}$ is a sequence of real numbers and let $\{M_i\}_{i\in\natural{}}$ be the sequence of surfaces
given by
$\big\{M_i:=M+(0,b_i,0)\big\}_{i\in\natural{}}.$
Then, after passing to a subsequence, $\{M_{i}\}_{i\in\natural{}}$ converges smoothly with multiplicity one to
a properly embedded connected translating soliton $M_{\infty}$ which has the same asymptotic behavior as $M$.
\end{lemma}

\begin{proof}
Recall that any translator $M\subset\real{3}$ can be 
regarded as a minimal surface of $(\Omega=\real{3},\operatorname{g})$ where $\operatorname{g}$ is the Ilmanen's metric.  Notice that each element
of the sequence $\{M_{i}\}_{i\in\natural{}}$ has the same asymptotic behavior as $M$.
Without loss of generality, we can arrange the coordinate system such that
$$\mathcal{C}=\big\{(x_1,x_2,x_3)\in\real{3}:x^2_1+x^2_3\le r_0^2\big\}.$$
By assumption our surface $M$ is $C^1$-asymptotic outside $\mathcal{C}$ to two half-planes $\mathcal{H}_1$, $\mathcal{H}_2$ (see Fig. \ref{tilted_planes}).
\begin{figure}[h]
\includegraphics[scale=0.40]{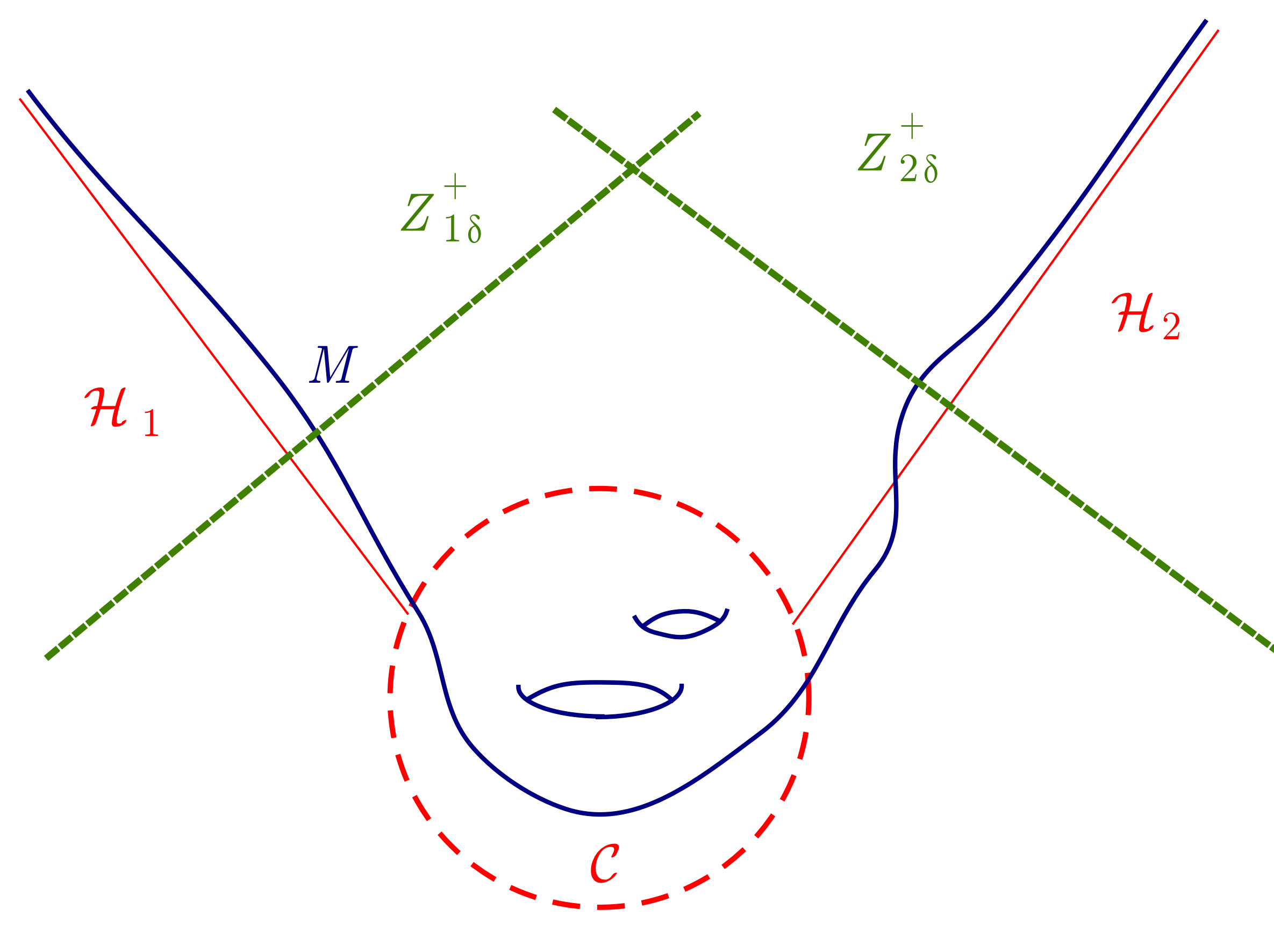}\caption{Asymptotic behaviour with tilted half-planes}\label{tilted_planes}
\end{figure}
Let now $\mbox{w}_1$,
$\mbox{w}_2$
be the unit inward pointing vectors of $\partial\mathcal{H}_1$, $\partial\mathcal{H}_2$, respectively. For any $\delta>0$ consider the closed half-planes
$$\mathcal{H}_{k}(\delta)=\{p+t\mbox{w}_k:p\in\partial\mathcal{H}_k\,\,\text{and}\,\, t\ge\delta\},$$
for $k\in\{1,2\}$ and denote by $Z^{+}_{k\delta}$, $k\in\{1,2\}$, the closed half-space of $\real{3}$ containing $\mathcal{H}_k(\delta)$ and
with boundary containing $\partial\mathcal{H}_k(\delta)$ and being perpendicular to $\mbox{w}_k$. Moreover, consider the closed half-spaces
$$Z^{-}_{k\delta}=\big(\real{3}-Z^{+}_{k\delta}\big)\cup\partial Z^{+}_{k\delta},$$
for any $k\in\{1,2\}$.

In the case where the sequence $\{b_i\}_{i\in\natural{}}$ is bounded, we can consider a subsequence such that $\lim b_i=b_{\infty}\in\real{}$.
Then obviously $\{M_i\}_{i\in\natural{}}$ converges smoothly with multiplicity one to the properly embedded translating soliton
$$M_{\infty}=M+(0,b_{\infty},0).$$
Clearly $M_{\infty}$ has the same asymptotic behavior with $M$.

Let us examine now the case where the sequence $\{b_i\}_{i\in\natural{}}$ is not bounded. Split each surface $M_i$ of the surface into the parts
$$M^{+}_{1i}(\delta):=M_i\cap Z^{+}_{1\delta},\,\,M^{+}_{2i}(\delta):=M_i\cap Z^{+}_{2\delta}\,\,\text{and}\,\,M^{-}_{i}(\delta):=M_i\cap Z^{-}_{1\delta}\cap Z^{-}_{2\delta}.$$

{\bf Claim 1.} {\it The sequences $\{M^{+}_{1i}(\delta)\}_{i\in\natural{}}$ and $\{M^{+}_{2i}(\delta)\}_{i\in\natural{}}$ have uniformly bounded area on compact sets.}

{\it Proof of the claim.} Let $K$ be a compact subset of $\Omega$ and $\mathbb{B}_r(0)$ a ball of radius $r$ centered at the origin of $\real{3}$ containing $K$.
Denote by $V_i$ the projection of the surface $M^{+}_{1i}(\delta)\cap K$ to the closed half-plane $\mathcal{H}_1(\delta)$. Hence we can parametrize $M^{+}_{1i}(\delta)$ by
a map $\Phi_i:V_i\to\real{3}$ of the form
\begin{eqnarray*}
\Phi_i(s,t)&=&(c_1,c_2,c_3)+s\mbox{e}_2+t\mbox{w}_1+\varphi(s-b_i,t)\mbox{e}_2\wedge\mbox{w}_1\\
&=&\big\{c_1+(\cos\alpha)t+(\sin\alpha)\varphi(s-b_i,t)\big\}\mbox{e}_1+\big\{c_2+s\big\}\mbox{e}_2\\
&&+\big\{c_3+(\sin\alpha)t-(\cos\alpha)\varphi(s-b_i,t)\big\}\mbox{e}_3,
\end{eqnarray*}
where $i\in\natural{}$, $\big\{\mbox{e}_1,\mbox{e}_2,\mbox{e}_3\big\}$ is the standard basis of $\real{3}$, $\alpha$ is the angle between the vectors $\mbox{e}_1$
and $\mbox{w}_1$ and $(c_1,c_2,c_3)$ is a fixed point on $\partial\mathcal{H}_{1}(\delta)$. By taking $\delta$ very large we can make sure that $|\varphi|$ and
$|D\varphi|$ are bounded by a universal constant $\varepsilon$. Hence, for any index $i\in\natural{}$ we have that
\begin{eqnarray*}
\mbox{area}_{\gind}\big\{M_{1i}^{+}(\delta)\cap K\big\}&=&\int_{V_i}e^{c_3+(\sin\alpha)t-(\cos\alpha)\varphi(s-b_i,t)}\sqrt{1+|D\varphi|^2}\,\,dsdt\\
&\le& \int_{V_i}e^{c_3+c(r)+\varepsilon}\sqrt{1+\varepsilon^2}\,\,dsdt\\
&=&e^{c_3+c(r)+\varepsilon}\sqrt{1+\varepsilon^2}\,\,\operatorname{area}_{\operatorname{euc}}(V_i),
\end{eqnarray*}
where $c(r)$ is a constant depending on $r$ and $\operatorname{area}_{\operatorname{euc}}(V_i)$ is the euclidean area of $V_i$. Note that
$\operatorname{area}_{\operatorname{euc}}(V_i)$ is less or equal than the euclidean area of the projection of $K$ to the plane containing $\mathcal{H}_1(\delta)$.
Thus there exists a number $m(K)$ depending only on $K$ such that
$$\mbox{area}_{\gind}\big\{M^{+}_{1i}(\delta)\cap K\big\}\le m(K).$$
Consequently, $\big\{M^{+}_{1i}(\delta)\big\}_{i\in\natural{}}$ has uniformly bounded area. Similarly, we show that $\big\{M^{+}_{2i}(\delta)\big\}_{i\in\natural{}}$ has uniformly
bounded area and this concludes the proof of the claim.

{\bf Claim 2.} {\it The sequence of surfaces $\big\{M^{-}_{i}(\delta)\big\}_{i\in\natural{}}$ has uniformly bounded area on compact sets.}

{\it Proof of the claim.} Let us show a first that the sequence $\big\{\partial M^{-}_{i}(\delta)\big\}_{i\in\natural{}}$ has uniformly bounded length on compact sets.
Following the notation introduced in the above claim, each connected component of $\partial M^{-}_i(\delta)$ can be represented as the image of the curve
$\gamma_i:\real{}\to\real{3}$ given by
\begin{eqnarray*}
\gamma_i(s)&=&\big\{c_1+(\cos\alpha)\delta+(\sin\alpha)\varphi(s-b_i,\delta)\big\}\mbox{e}_1\\
&&+\big\{c_2+s\big\}\mbox{e}_2
+\big\{c_3+(\sin\alpha)\delta-(\cos\alpha)\varphi(s-b_i,\delta)\big\}\mbox{e}_3,
\end{eqnarray*}
for any index $i\in\natural{}$. Let $K$ be a compact set of $\Omega$, $\mathbb{B}_r(0)$ a ball of radius $r$ centered at the origin and containing $K$.
Denote by $I_i$ the projection of $\partial M^{-}_i(\delta)\cap K$
to $\partial\mathcal{H}_1(\delta)$. Estimating as in Claim 1, we get that
\begin{eqnarray*}
\operatorname{length}_{\gind}\big\{\partial M^{-}_{i}(\delta)\cap K\big\}
\le\int_{I_i}e^{\tfrac{c_3+c(r)+\varepsilon}{2}}\sqrt{1+\varepsilon^2}\,\,ds,
\end{eqnarray*}
where $c(r)$ is a constant depending on $r$.
Thus, there exists a constant $n(K)$ depending only on the compact set $K$ such that
$$\operatorname{length}_{\gind}\big\{\partial M^{-}_{i}(\delta)\cap K\big\}\le n(K).$$
Hence, the sequence $\big\{\partial M^{-}_{i}(\delta)\big\}_{i\in\natural{}}$ has uniformly bounded length on compact sets.

Recall now that the set $\mathscr{Z}$ is closed. From Claim 1 it
follows that $\mathscr{Z}$ is contained inside a cylinder.
Consider now a translating paraboloid and translate it in the direction
of the $x_3$-axis until it has no common point with $\mathscr{Z}$.
Then move back the translating paraboloid until it intersects
for the first time the set $\mathscr{Z}$ (see Fig. \ref{areablow}).
\begin{figure}[h]
\includegraphics[scale=.08]{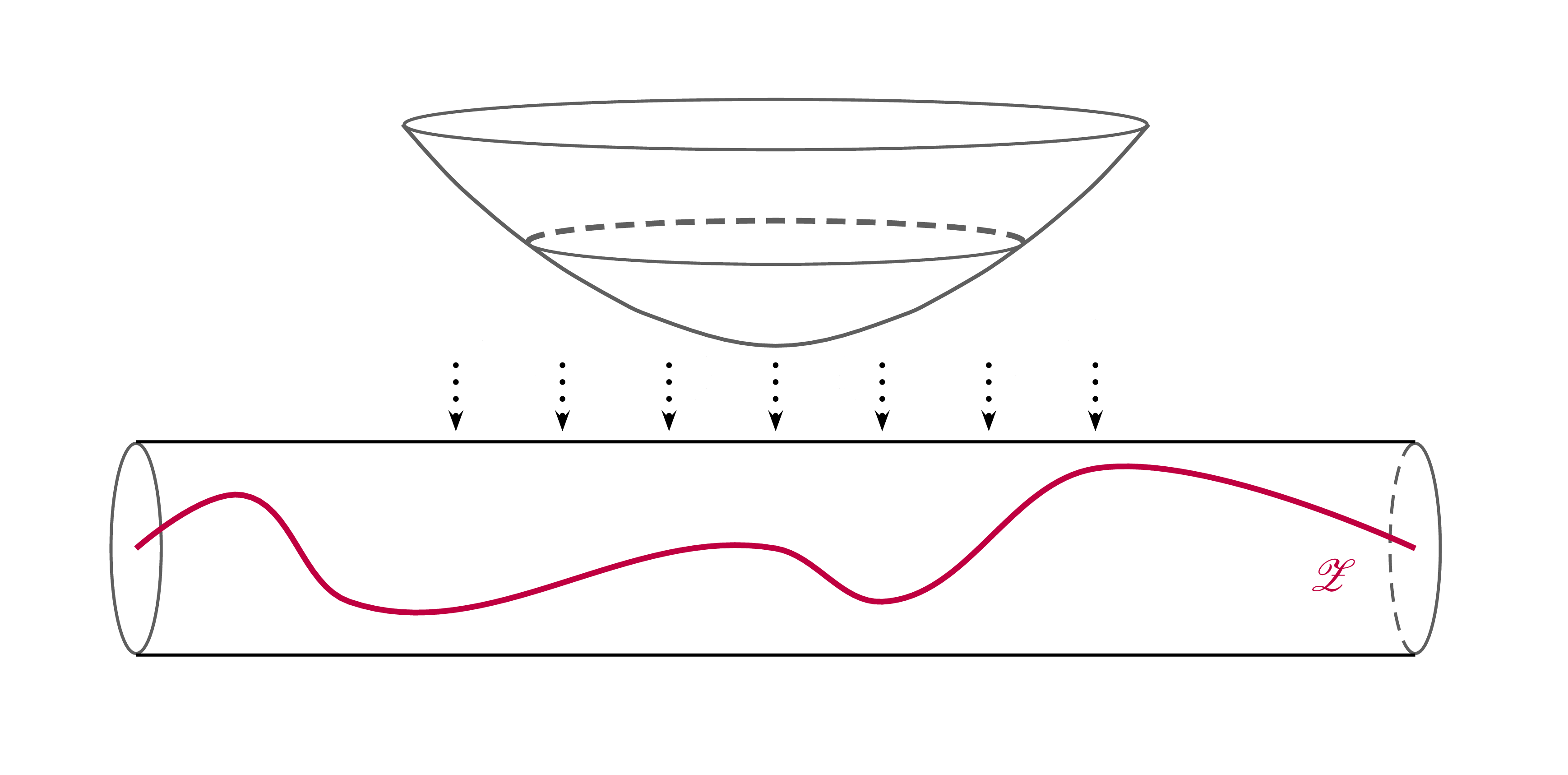}\caption{The area blow-up set $\mathscr{Z}$}\label{areablow}
\end{figure}
From the strong
barrier principle of White (Theorem \ref{thm:Controlling_area-blowup}), the translating paraboloid is contained in $\mathscr{Z}$.
But this leads to a contradiction, because now the area blow-up set $\mathscr{Z}$ is not contained inside a cylinder. Thus, $\mathscr{Z}$ must be empty and consequently
$\{{M}^-_i(\delta)\}_{i\in\natural{}}$ has uniformly bounded area.
 
Since the parts $\{{M}^+_{1i}(\delta)\}_{i\in\natural{}}$, $\{{M}^+_{2i}(\delta)\}_{i\in\natural{}}$, $\{{M}^-_{i}(\delta)\}_{i\in\natural{}}$
have uniformly bounded area, we see that the whole sequence $\{{M}_{i}\}_{i\in\natural{}}$
has uniformly bounded area. From our assumptions, also the genus
of the sequence  is uniformly bounded. The convergence to a smooth properly embedded translator
$M_{\infty}$ follows from Theorem \ref{thm:compactness_theorem_Brian_White} of White.
Since each ${M}^+_{ki}(\delta)$, $k\in\{1,2\}$, is a graph and each $M_i$ is connected, we deduce that
the multiplicity is one everywhere. Thus, the convergence is smooth.
Moreover, observe that each component of $M_{\infty}\cap Z_{k\delta}^+$, $k\in\{1,2\}$, can be represented
as the graph of a smooth function $\varphi_{\infty}$ which is the limit of the sequence of graphs generated
by the smooth functions
$$\varphi_i (s,t)=\varphi(s-b_i,t)$$
for any $i\in\natural{}$. Thus, the limiting surface $M_{\infty}$ has the same 
asymptotic behavior as $M$.
The limiting surface $M_{\infty}$ must be connected since otherwise there should exist a 
properly embedded
connected component $\Sigma$ of $M$ lying inside $\mathcal{C}$. But then,
the $x_3$-coordinate function of $\Sigma$ must be bounded from above, which is absurd.
This concludes the proof.
\end{proof}

As a first application of the above compactness result we show that the half-planes $\mathcal{H}_1$
and $\mathcal{H}_2$ must be parallel to each other.

\begin{lemma}
Let $M$ be a translating soliton as in our theorem. Then, the half-planes $\mathcal{H}_1$ and $\mathcal{H}_2$
must be parallel to the translating direction. Moreover, if $\mathcal{H}_1$ and $\mathcal{H}_2$
are parts of the same plane $\Pi$, then $M$ should coincide with $\Pi$.
\end{lemma}

\begin{proof}
We follow the notation introduced in the last lemma. Assume to the contrary that the half-plane
$$\mathcal{H}_1=\big\{p+t\operatorname{w}_1:p\in\partial\mathcal{H}_1\,\,\text{and}\,\,t>0\big\}$$
is not parallel to the translating direction $\v$. Let us suppose at first that
the cosine of angle between the unit inward pointing normal $\operatorname{w}_1$ of $\partial\mathcal{H}_1$
and $\mbox{e}_1$ is positive. Consider the strip $S_{t_0}$ given by
$$S_{t_0}:=\left(t_0-{\pi}/{2},t_0+{\pi}/{2}\right)\times \real{} \times \real{}.$$
For sufficiently large $t_0$ this slab does not intersects the cylinder $\mathcal{C}$.
For fixed real numbers $t,l$ let $\gp^{t,l}$ be the grim reaper cylinder
$$\gp^{t,l}:=\big\{(x_1,x_2,l+\log\cos(x_1-t))\in\real{3}:|x_1-t|<{\pi}/{2}, x_2\in\real{}\big\}.$$
By our assumptions, as $\delta$ becomes larger the wing
$M_{\delta}:=M\cap Z^{+}_{1\delta}$
of $M$ is getting closer to $\mathcal{H}_1$. By the asymptotic behavior of $M$ to two half-planes, there exists $t_0,l_0\in \real{}$ large enough such
that $\gp^{t_0,l_0}$ does not intersect $M_{\delta}$. Then translate this grim reaper cylinder in the direction of $-\v$. Since $\mathcal{H}_1$ is not parallel to $\v$, after
some finite time $l_1$ either there will be a first interior point of contact between the surface $M_{\delta}$ and $\gp^{t_0,l_0-l_1}$ or there will exist
a sequence of points
$\{p_i=(p_{1i},p_{2i},p_{3i})\}_{i\in\natural{}}$ in the interior of $M_{\delta}$, with $\{p_{3i}\}_{i\in\natural{}}$ bounded and $\{p_{2i}\}_{i\in\natural{}}$
unbounded, such that
$$\lim_{i\to\infty}\operatorname{dist}(p_i,\gp^{t_0,l_0-l_1})= 0.$$
The first possibility contradicts the asymptotic behavior of $M$. So let us examine the second possibility. Consider the
sequence of surfaces $\{M_i\}_{i\in\natural{}}$ given by
$M_i=M+(0,-p_{2i},0),$
for any $i\in\natural{}$. By Lemma \ref{pro:limit_process}, after passing to a subsequence, $\{M_i\}_{i\in\natural{}}$
converges smoothly to a connected and properly embedded translator $M_{\infty}$ which has the same asymptotic behavior as $M$.
But now there exists an interior point of contact between $M_{\infty}$ and $\gp^{t_0,l_0-l_1}$, which is
absurd. Similarly we treat the case where the cosine of the angle between $\operatorname{w}_1$ and $\operatorname{e}_1$ is negative. Hence both half-planes must be
parallel to the translating direction $\v$.

Suppose now that the half-planes $\mathcal{H}_1$ and $\mathcal{H}_2$ are contained in the same vertical plane $\Pi$.
Without loss of generality we may assume that $\Pi=\Pi(0)$. Suppose to the contrary that the translator $M$ does not coincide with $\Pi$.
Observe that in this case the $x_1$-coordinate function attains a non-zero supremum or a non-zero infimum
along a sequence $\{p_i=(p_{1i},p_{2i},p_{3i})\}_{i\in\natural{}}$ in the interior of $M$, with $\{p_{3i}\}_{i\in\natural{}}$ bounded
and $\{p_{2i}\}_{i\in\natural{}}$ unbounded. Performing a limiting process as in the previous case we arrive to a contradiction.
Therefore, the $x_1$-coordinate function must be zero constant and thus $M$ must be planar.
\end{proof}

Another application of the above compactness result is the following strong maximum principle.
\begin{lemma}\label{max}
Let $M$ be a translating soliton as in our theorem and assume that the half-planes $\mathcal{H}_1$ and $\mathcal{H}_2$ are distinct.
Consider a portion $\Sigma$ of $M$ (not necessarily compact) with non-empty boundary $\partial{\Sigma}$ such that the
$x_3$-coordinate function of $\Sigma$ is bounded. Then the supremum and the infimum of the $x_1$-coordinate function of
$\Sigma$ are reached along the boundary of $\Sigma$ i.e., there exists no sequence $\{p_i\}_{i\in\natural{}}$ in the interior
of $\Sigma$ such that ${\lim}_{i\to\infty}\dist(p_i,\partial\Sigma)>0$ and
${\lim}_{i\to\infty}x_1(p_i)={\sup}_{\Sigma}x_1$ or ${\lim}_{i\to\infty}x_1(p_i)={\inf}_{\Sigma}x_1.$
\end{lemma}

\begin{proof}
Recall that from the above lemma the half-planes $\mathcal{H}_1$ and $\mathcal{H}_2$ must be parallel to each other
and to the direction $\v$ of translation.
From our assumptions the $x_1$-coordinate function of the surface $M$ is bounded. Moreover,
the extrema of $x_1$ cannot be attained at an interior point of $\Sigma$, since otherwise from the tangency principle
$\Sigma$ should be a plane. This would imply that $M$ is a plane, something that contradicts
the asymptotic assumptions. So, let us suppose that there exists a sequence of points $\{p_i=(p_{1i},p_{2i},p_{3i})\}_{i\in\natural{}}$
in the interior of $\Sigma$ such that $\lim_{i\to\infty} \mbox{dist}(p_i,\partial{\Sigma})>0$ and $x_1(p_i)$ is tending to its
supremum or infimum. Then, consider the
sequence of surfaces $\{M_i\}_{i\in\natural{}}$ given by
$M_i=M+(0,-p_{2i},0),$
for any $i\in\natural{}$. By Lemma \ref{pro:limit_process}, after passing to a subsequence, $\{M_i\}_{i\in\natural{}}$
converges smoothly to a connected and properly embedded translator $M_{\infty}$ which has the same asymptotic behavior as $M$.
But now there exists a point in $M_{\infty}$ where its $x_1$-coordinate function reaches its local extremum, which is
absurd.
\end{proof}

\begin{remark}
The $x_1$-coordinate function of $M$ satisfies the partial differential equation
$\Delta x_1+\langle\nabla x_1,\nabla x_3\rangle=0.$
However, Lemma \ref{max} is not a direct consequence of the strong maximum principle for elliptic
PDE's because in general $\Sigma$ is not bounded.
\end{remark}

%
%

\section {Proof of the theorem}\label{sec:proof_of_the_main_theorem}
We have to deal only with the case where $\mathcal{H}_1$ and $\mathcal{H}_2$ are distinct and parallel to $\v$.
We can arrange
the coordinates such that $\v=(0,0,1)$ and such that the $x_2$-axis is the axis of rotation of our cylinder
$$\mathcal{C}=\big\{(x_1,x_2,x_3)\in\real{3}:x^2_1+x^2_3\le r^2\big\}.$$
Following the setting in \cite{fra14} let us define the family of planes $\{\Pi(t)\}_{t\in\real{}}$,
given by
$$\Pi(t):=\big\{(x_1,x_2,x_3)\in\real{3}:x_1=t\big\}.$$
Moreover, given a subset $A$ of $\real{3}$, for any $t\in \real{}$ we define the sets
\begin{eqnarray*}
A_+(t)&:=&\big\{(x_1,x_2,x_3)\in A:x_1\ge t\big\},\\
A_-(t)&:=&\big\{(x_1,x_2,x_3)\in A:x_1\le t\big\},\\
A^+(t)&:=&\big\{(x_1,x_2,x_3)\in A:x_3\ge t\big\},\\
A^-(t)&:=&\big\{(x_1,x_2,x_3)\in A:x_3\le t\big\},\\
A^*_+(t)&:=&\big\{(2t-x_1,x_2,x_3)\in\real{3}:(x_1,x_2,x_3)\in A_+(t)\big\},\\
A^*_-(t)&:=&\big\{(2t-x_1,x_2,x_3)\in\real{3}:(x_1,x_2,x_3)\in A_-(t)\big\}.
\end{eqnarray*}

Note that $A^*_+(t)$ and $A^*_-(t)$ are the image of $A_+(t)$ and $A_-(t)$ by the reflection respect to the plane $\Pi(t)$.

{\bf STEP 1:} We claim that both parts of $M$ outside the cylinder point
in the direction of $\v$. We argue indirectly. Let us suppose that one part of
$M-\mathcal{C}$ is asymptotic to
$$\mathcal{H}_1=\big\{(x_1,x_2,x_3)\in\real{3}:x_3> r_1>0,\, x_1=-\delta\big\}$$
and the other part is asymptotic to
$$\mathcal{H}_2=\big\{(x_1,x_2,x_3)\in\real{3}:x_3< r_2<0,\, x_1=+\delta\big\},$$
for some $\delta>0$ (see Fig. \ref{comp-left}).
\begin{figure}[h]
\includegraphics[scale=.07]{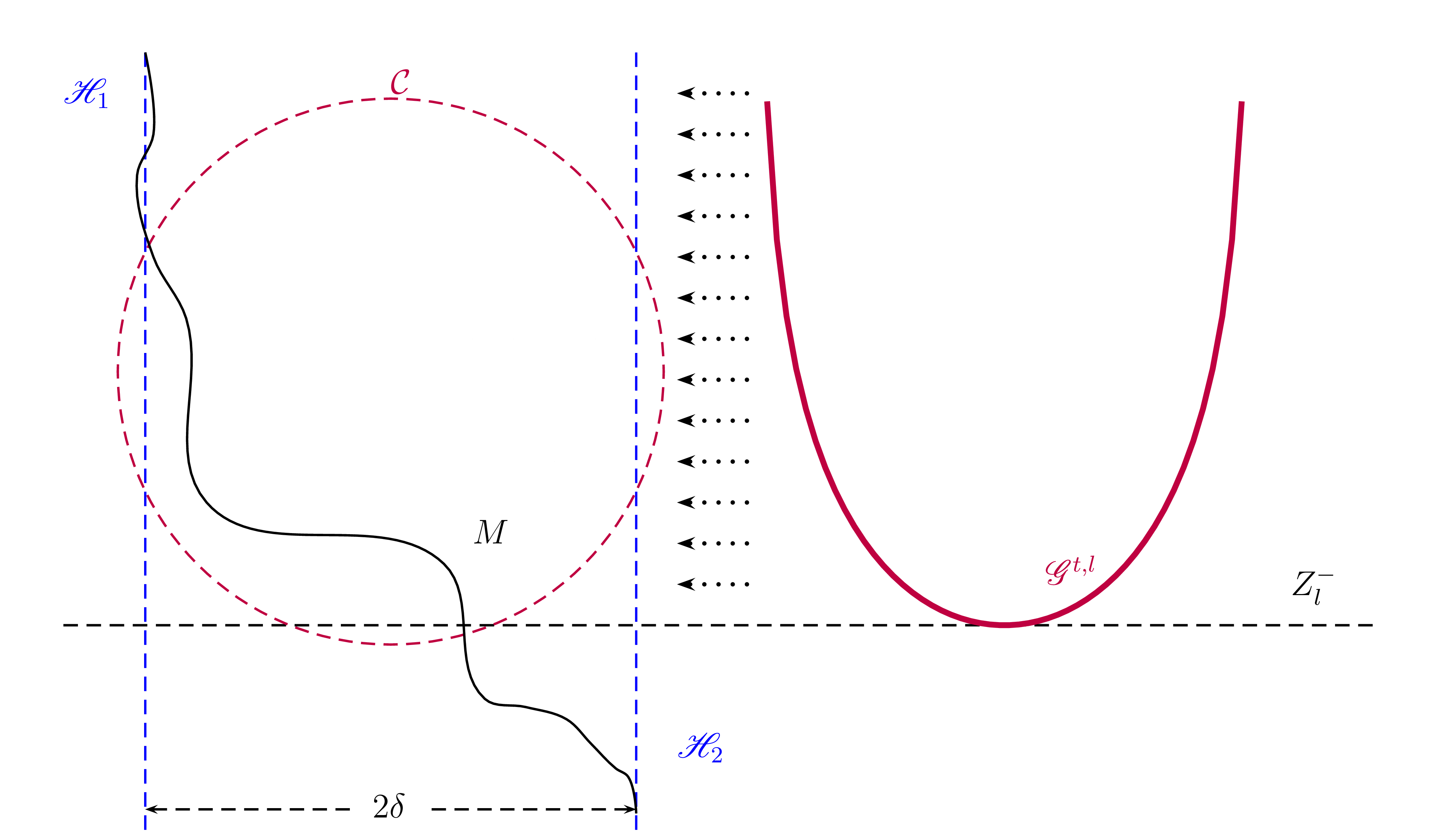}\caption{Comparison with a grim reaper cylinder}\label{comp-left}
\end{figure}
Fix real numbers $t,l$ and let $\gp^{t,l}$ be the grim reaper cylinder
$$\gp^{t,l}:=\big\{(x_1,x_2,l+\log\cos(x_1-t))\in\real{3}:|x_1-t|<{\pi}/{2}, x_2\in\real{}\big\}.$$
The idea is to obtain a contradiction by comparing the surface $M$ with an appropriate grim reaper cylinder
$\gp^{t,l}$. Let us start with the grim reaper cylinder $\gp^{{\pi}/{2}+\delta,0}.$ Note that
$\gp^{{\pi}/{2}+\delta,0}$ lies outside the strip $(-\delta,\delta)\times\real{2}$ and it is
asymptotic to two half-planes contained in $\Pi(\delta)$ and $\Pi(\delta+\pi)$.
\begin{figure}[h]
\includegraphics[scale=.09]{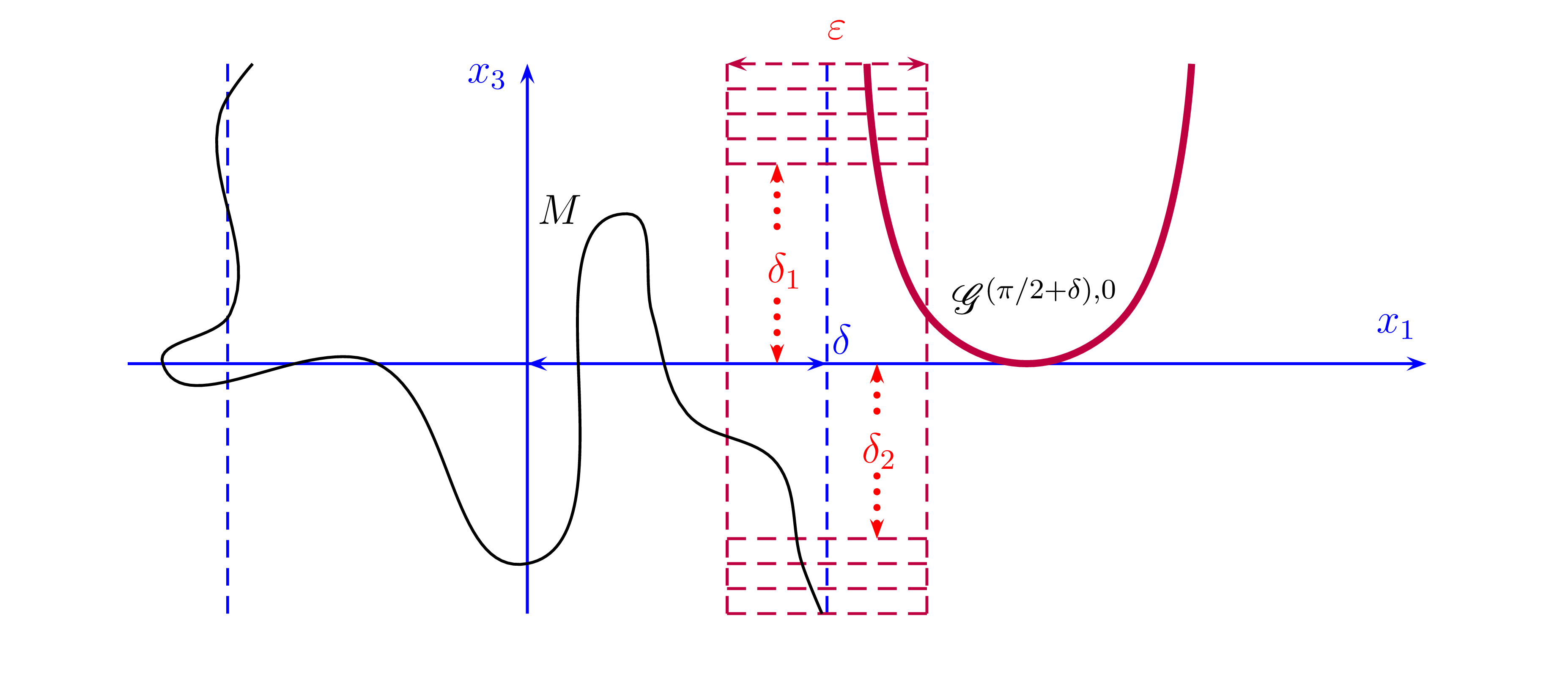}\caption{Comparison with a grim reaper cylinder}\label{comp-left-1}
\end{figure} 

Fix $\varepsilon\in(0,2\delta)$.
Because outside a cylinder the grim reaper cylinder $\gp^{{\pi}/{2}+\delta,0}$ is asymptotic to two half-planes, there exists
$\delta_1>0$, depending on $\varepsilon$, such that $\gp^{{\pi}/{2}+\delta,0}\cap Z^+_{\delta_1}$ is inside the region
$$(\delta,\delta+\varepsilon/2)\times\real{}\times (\delta_1,+\infty).$$
Moreover, there exists $\delta_2>0$, depending on $\varepsilon$, such that $M\cap Z^-_{-\delta_2}$ is inside the region
$$(\delta-\varepsilon/2,\delta+\varepsilon/2)\times\real{}\times (-\infty,-\delta_2).$$
Consider now the grim reaper cylinder $\gp^{{\pi}/{2}+\delta+t,-\delta_1-\delta_2-1}$
and choose $t$ large enough so that
$$\gp^{{\pi}/{2}+\delta+t,-\delta_1-\delta_2-1}\cap M=\emptyset.$$
Translate the above grim reaper cylinder in the direction of $(-1,0,0)$. Since $\varepsilon\in (0,2\delta)$, we see that after some finite time $t_0$
either there will be a first interior point of contact between $M$ and $\gp^{{\pi}/{2}+\delta+t_0,-\delta_1-\delta_2-1}$ or there will exist a sequence
$\{p_i=(p_{1i},p_{2i},p_{3i})\}_{i\in\natural{}}$ of points in $M$, with $\{p_{3i}\}_{i\in\natural{}}$ bounded and $\{p_{2i}\}_{i\in\natural{}}$
unbounded, such that
$$\lim_{i\to\infty}\operatorname{dist}(p_i,\gp^{{\pi}/{2}+\delta+t_0,-\delta_1-\delta_2-1})= 0.$$
As in Lemma \ref{max}, we deduce that both cases contradict the asymptotic behavior of $M$.
Therefore, both parts of $M-\mathcal{C}$ must point in the direction of $\v$.

{\bf STEP 2:} We claim now that $M$ lies in the slab
$S:=\big(-\delta,+\delta\big)\times\real{2}.$
Assume at first that $\lambda:={\sup}_Mx_1>\delta.$ Consider now the surface (see Fig. \ref{comp-plane})
$$\Sigma:=\{(x_1,x_2,x_3)\in M:x_1\ge\delta/2+\lambda/2\}.$$
\begin{figure}[h]
\includegraphics[scale=.1]{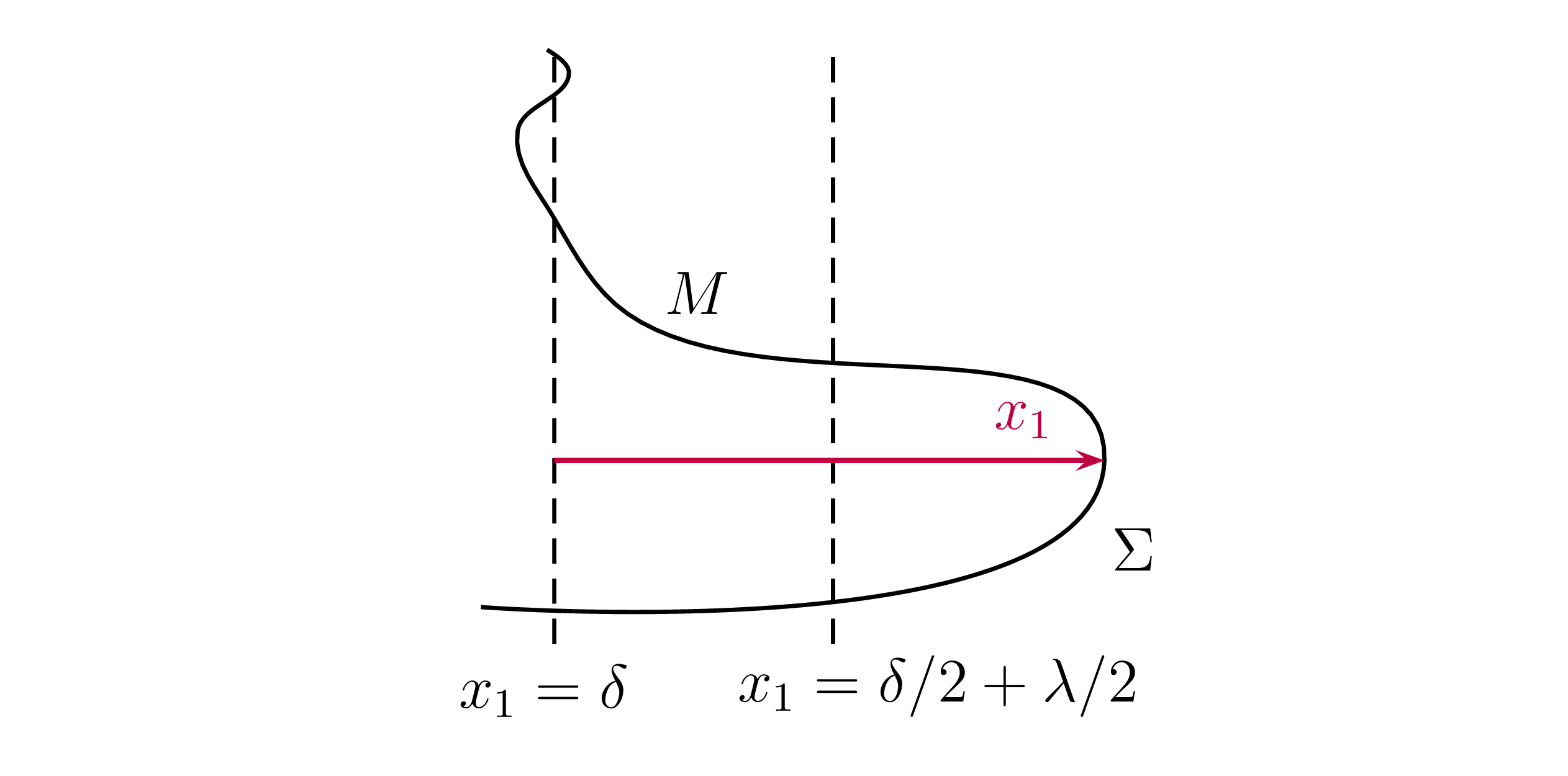}\caption{A slice of $\Sigma$}\label{comp-plane}
\end{figure}
The asymptotic assumptions on $M$ imply that the $x_3$-coordinate of $\Sigma$
is bounded. Therefore, due to Lemma \ref{max},
$${\sup}_{\Sigma}x_1={\sup}_{\partial{\Sigma}}x_1.$$
But since
$$\partial\Sigma\subset\{(x_1,x_2,x_3)\in\real{3}:x_1=\delta/2+\lambda/2\},$$
we have that
$$x_1(p)={\delta}/{2}+{\lambda}/{2}<\lambda={\sup}_\Sigma{x_1},$$
for any $p\in\partial\Sigma$, which is absurd. Thus $\sup_{M}x_1\le\delta$. Observe that if equality holds, then a contradiction is reached comparing $M$
and the plane $\Pi(\delta)$ using the tangency principle. Hence $\sup_{M}x_1<\delta$. Similarly, we can prove that ${\inf}_{M} x_1>-\delta.$
Consequently, $M$ should lie inside the slab $S$.

{\bf STEP 3:} Using the same arguments we will prove now that $2\delta=\pi$. Indeed, suppose at first that $2\delta>\pi$.
We can then place a grim reaper cylinder $\gp^{0,l}$ inside the slab $S$, by taking $l$ sufficiently large, so that $\gp^{0,l}\cap M=\emptyset$
(see Fig. \ref{comp-1}).
\begin{figure}[h]
\includegraphics[scale=.09]{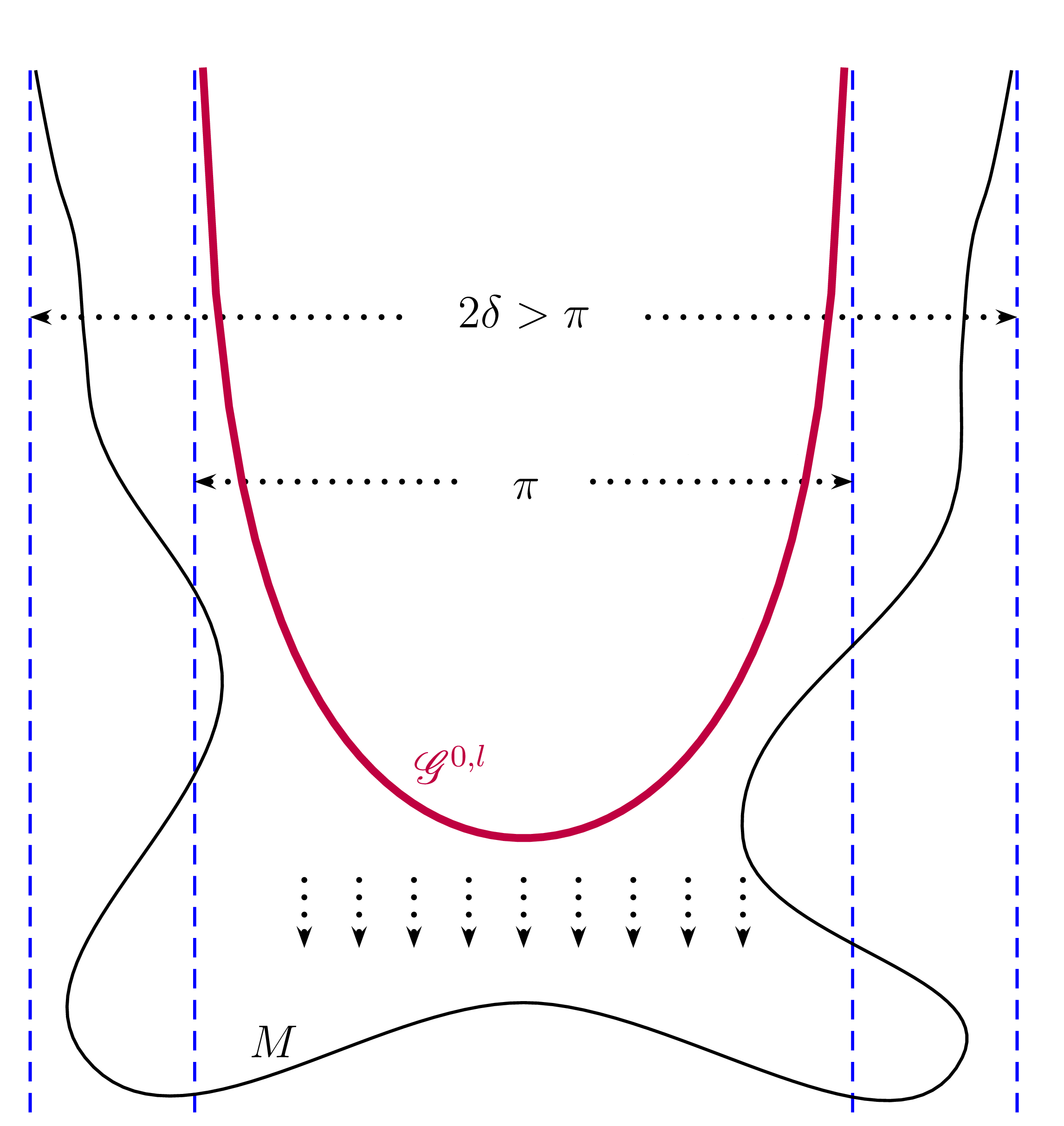}\caption{Comparison with a grim reaper cylinder from inside}\label{comp-1}
\end{figure}
Consider now the set
$$\mathscr{A}:=\big\{l> 0: M\cap \gp^{0,l} = \emptyset  \}.$$
Let $l_0:=\inf\mathscr{A}$.
Assume at first that $l_0\notin\mathscr{A}$. Because
$M\cap\gp^{0,l_0}\neq\emptyset,$ it follows that there is an interior point of contact between $M$ and $\gp^{0,l_0}$.
But then $M\equiv\gp^{0,l_0}$ which leads to a contradiction with the asymptotic assumptions on $M$.
Let us treat now the case where $l_0\in\mathscr{A}$. In this case
$\dist\big\{M,\gp^{0,l_0}\big\}=0.$
Therefore, there exists a sequence of points
$\big\{p_i=(p_{1i},p_{2i},p_{3i})\}_{i\in\natural{}}$ in $M$
such that
$$\lim_{i\to\infty}p_{1i}=p_{1\infty}\in\real{},\,\lim_{i\to\infty}p_{2i}=\infty,\,\lim_{i\to\infty}p_{3i}=p_{3\infty}\in\real{}$$
and
$$\lim_{i\to\infty}\operatorname{dist}\big(p_i,\gp^{0,l_0}\big)=0.$$
Consider the sequence
$$\big\{M_i=M+(0,-p_{2i},0)\big\}_{i\in\natural{}}.$$
By Lemma \ref{pro:limit_process} we know that after passing to a subsequence,
$\{M_i\}_{i\in\natural{}}$ converges to a connected properly embedded translator $M_{\infty}$ which has the same asymptotic
behavior as $M$. On the other hand $M_{\infty}$ has an interior point of contact
with $\gp^{0,l_0}$ and thus they must coincide. But this contradicts again the assumption on the asymptotic
behavior of $M$. Thus $2\delta$ must be less or equal than $\pi$. We exclude also the case where
$2\delta<\pi$ by comparing $M$ with a grim reaper cylinder from outside (see Fig. \ref{comp-2}).
\begin{figure}[h]
\includegraphics[scale=.07]{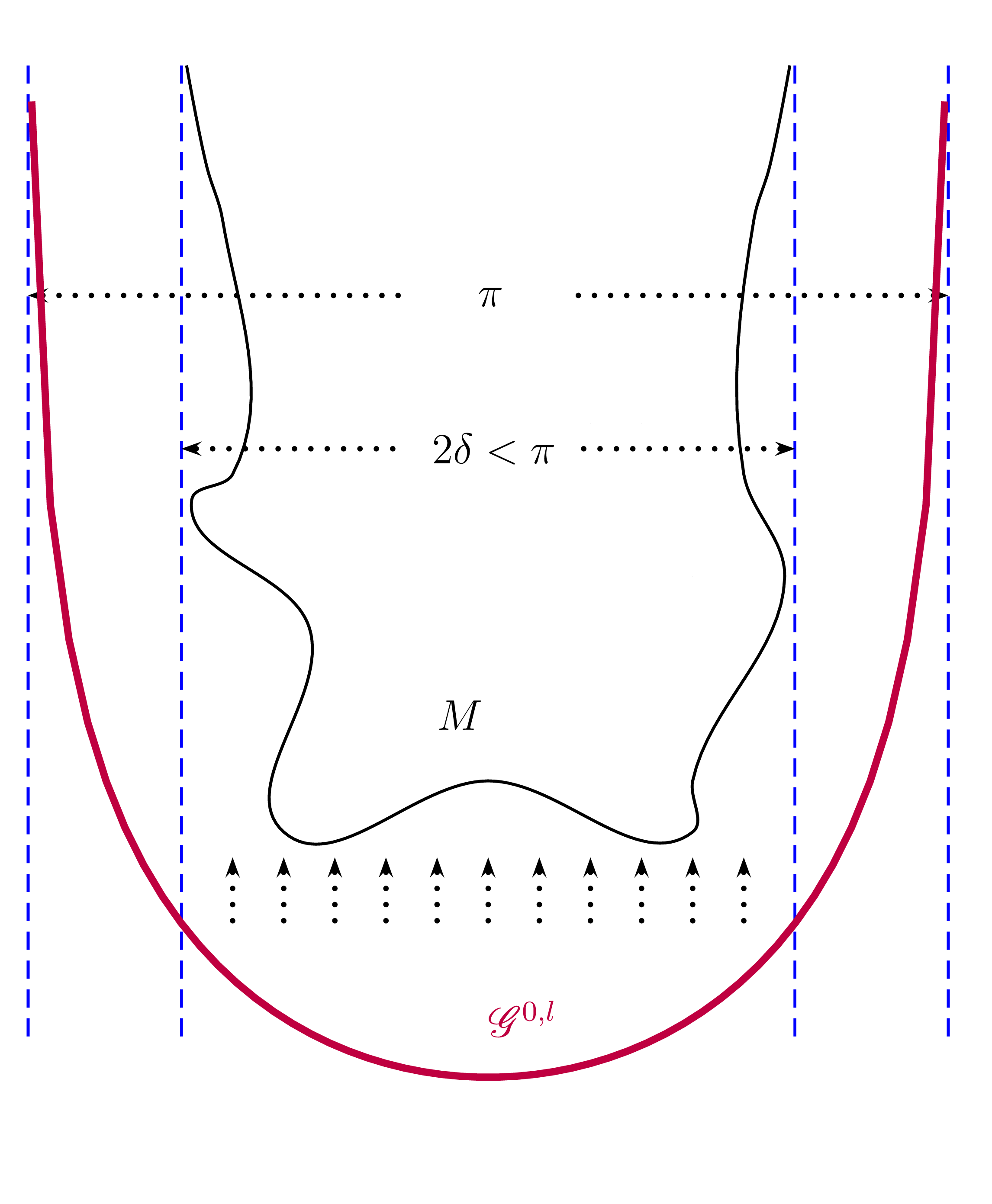}\caption{Comparison with a grim reaper cylinder from outside}\label{comp-2}
\end{figure}
Consequently, $2\delta=\pi.$

{\bf STEP 4:} We will prove here two auxiliary results that will be very useful in the rest of the proof.
\begin{claim}\label{wm}
The inequality
$$-\pi/2<{\inf}_{\partial M^-(t)}x_1\le{\inf}_{M^-(t)}x_1
\le{\sup}_{M^-(t)}x_1\le{\sup}_{\partial M^-(t)}x_1<\pi/2,$$
holds for any any real number $t$ such that $M^{-}(t)\neq\emptyset$.
\end{claim}

{\it Proof of the claim.} Recall that
$$M^-(t)=\{(x_1,x_2,x_3)\in M:x_3\le t\}.$$
Hence, from Lemma 3.2, we have that
$$\dist\left(M^-(t),\Pi({\pi}/{2})\right)=\dist\left(\partial{M^-(t)},\Pi({\pi}/{2})\right).$$
Suppose now to the contrary that
$$\dist\left(\partial{M^-(t)},\Pi({\pi}/{2})\right)=0.$$
Then, there exists a sequence $\{p_i=(p_{1i},p_{2i},t)\}_{i\in\natural{}}$ of points of $\partial{M^-(t)}$ such that
$$\lim_{i\to\infty}p_{1i}=\pi/2\quad \text{and}\quad \lim_{i\to\infty}p_{2i}=\infty.$$
Consider the sequence
of surfaces $\{M_i:=M+(0,-p_{2i},0)\}_{i\in\natural{}}$. From Lemma 3.1 we know that $\{M_i\}_{i\in\natural{}}$
converges to a connected properly embedded translator $M_{\infty}$ which has the same asymptotic behavior
as $M$. On the other hand, there is an interior point of contact between $M_{\infty}$ and $\Pi(\pi/2)$, which is
a contradiction. Thus,
$$\dist\left(\partial{M^-(t)},\Pi({\pi}/{2})\right)>0.$$
which implies that
${\sup}_{M^-(t)}x_1<\pi/2.$
In the same way, we can prove that
${\inf}_{M^-(t)}x_1>-\pi/2.$
This completes the proof of the claim.

\begin{claim}\label{graph-1}
There exists a sufficiently large number $t$ such that the parts of $M^+(t)$ are graphs over the $x_1x_2$-plane,
and there exists a sufficiently small $\delta>0$ such that $M_+({\pi}/{2}-\delta)$ is a graph over the $x_1x_2$-plane.
\end{claim}

{\it Proof of the claim.} From STEP 3 we know that $M$ lies inside the
slab
$$S=\left(-{\pi}/{2},{\pi}/{2}\right)\times \real{2}.$$
Since $\gp$ and $M-\mathcal{C}$ are $C^1$-asymptotic to $\Pi(\frac{\pi}{2})$, we can
represent each wing of $M-\mathcal{C}$ as a graph over $\gp$. Fix a sufficiently small positive
number $\varepsilon.$ Then, there exists $\delta>0$ such that the interior of the
right wing $M_+(\pi/2-\delta)$ of $M-\mathcal{C}$ can be parametrized by a smooth map
$f:T_{\delta}:=(\pi/2-\delta,\pi/2)\times\real{}\to\real{3}$ given by
$$f=u+\varphi\xi_{u},$$
where the map
$u(x_1,x_2)=(x_1,x_2,-\log\cos x_1)$
describes the position vector of $\gp$,
$\xi_{u}(x_1,x_2)=(\sin x_1,0,-\cos x_1)$
is the outer unit normal of $u$ and $\varphi:(\pi/2-\delta,\pi/2)\times\real{}\to\real{}$ is a smooth function such
that
$${\sup}_{T_{\delta}}|\varphi|<\varepsilon\quad\text{and}\quad{\sup}_{T_{\delta}}|D\varphi|<\varepsilon.$$
A straightforward computation shows that the outer unit normal $\xi$ of $f$ is given by the formula
\begin{equation}\label{gm}
\xi=\frac{(1+\varphi\cos x_1)\xi_{u}-(1+\varphi\cos x_1)\varphi_{x_2}u_{x_2}-\varphi_{x_1}\cos^2 x_1 u_{x_1}}
{\sqrt{(1+\varphi\cos x_1)^2(1+\varphi^2_{x_2})+\varphi^2_{x_1}\cos^2x_1}}.
\end{equation}
Because $f$ is a translator, we deduce that its mean curvature is
\begin{equation}\label{gaussmap}
H=-\langle\xi,\v\rangle=\frac{\cos x_1(1+\varphi\cos x_1+\varphi_{x_1}\sin x_1)}
{\sqrt{(1+\varphi\cos x_1)^2(1+\varphi^2_{x_2})+\varphi^2_{x_1}\cos^2 x_1}}.
\end{equation}
Consequently, $\langle\xi,\v\rangle<0$. Thus, each point of $M_+(\pi/2-\delta)$
has an open neighborhood that can be represented as a graph over the $x_1x_2$-plane.
Due to Lemma \ref{max}, the surface $M_+(\pi/2-\delta)$ must be connected. Indeed, assume to the contrary that $M_+(\pi/2-\delta)$
has more than one connected component. Let $\Sigma$ be a connected component different from the one whose
$x_3$-coordinate function is not bounded (there is at least one by assumption). Then due to Lemma \ref{max} the
infimum and the supremum of the $x_1$-coordinate function of $\Sigma$ are reached along the boundary, that is, $\Sigma$
is an open piece of the plane $\Pi(\pi/2-\delta)$, so the whole surface $M$ must coincide with this plane, which is
a contradiction. Moreover, its projection to the $x_1x_2$-plane must be the simply connected set $T_{\delta}$. Thus, $M_+(\pi/2-\delta)$
must be a global graph over the subset $T_{\delta}$ of the $x_1x_2$-plane. Similarly, we prove that also the left hand side wing of
$M-\mathcal{C}$ is graphical. This completes the proof of the claim because by the hypothesis on the asymptotic behavior
of $M$, there exists a sufficiently large number $t$ such that $M^+(t)\subset M_{-}(-\pi/2+\delta) \cup M_+(\pi/2-\delta)$.

{\bf STEP 5:} We shall prove now that $M$ is symmetric with respect to
$$\Pi(0)=\big\{(x_1,x_2,x_3)\in\real{3}:x_1=0\big\}$$
and that $M$ is a bi-graph over this plane.
The main tool used in the proof is the method of moving planes of Alexandrov (see \cite{alexandrov,schoen}).
Let us define
$$\mathcal{A} := \{ t\in [0,\pi/2): M_{+}(t) \text{ is a graph over } \Pi(0) \text{ and } M_{+}^{*}(t)\geq M_{-}(t) \}.$$
Recall from \cite[Definition 3.1]{fra14} that the relation $M_{+}^{*}(t)\geq M_{-}(t)$ means that $M_{+}^{*}(t)$ is on the right hand side of $M_-(t)$.
We will prove that $0\in \mathcal{A}$. In this case we have that $M_{+}^{*}(0)\geq M_{-}(0)$.
By a symmetric argument we can show that $M_{+}(0)\geq M_{-}^{*}(0)$. Thus $M^*_{+}(0)\equiv M_{-}(0)$ and the proof of this step
will be completed. The steps of the proof are the same as in \cite[Proof of Theorem A]{fra14} with the
difference that here we have to control the behavior of the Gau{\ss} map in the direction of the $x_2$-axis.

\begin{claim}
The minimum of the set $\mathcal{A}$ is $0$. In particular, $\mathcal{A}=[0,\pi/2)$.
\end{claim}
{\it Proof of the claim.} Due to Claim \ref{graph-1} it follows that given a sufficiently small number $\varepsilon$,
there exists a positive number $t$ such that the surface $M_+(t)$ can be represented as a graph over $\Pi(0)$ as well as a graph
over the $x_1x_2$-plane. Hence one can easily show that $\mathcal{A}$ is a non-empty set. Following the same arguments as
in \cite[Section 3, Proof of Theorem A]{fra14}, we can show that $\mathcal{A}$ is a closed subset of $[0,\pi/2)$. Moreover if $s\in \mathcal{A}$,
then $[s,\pi/2)\subset \mathcal{A}$. Suppose now that $s_{0}:=\min \mathcal{A}>0$. Then we will get at a contradiction, i.e., we will show
that there exists a positive number $\varepsilon$ such that $s_{0}-\varepsilon \in \mathcal{A}$.

\textit{Condition $1$}: We will show at first that there exists a positive constant $\varepsilon_{1}< s_{0}$ such that $M_{+}(s_{0}-\varepsilon_{1})$
is a graph over the plane $\Pi(0)$. Take a positive number $\alpha$ and consider the sets
$${M}^+_+(s):=\{(x_{1},x_2,x_{3})\in M_+(s): x_{3}>\alpha \},$$
$${M}^+_-(s):=\{(x_{1},x_2,x_{3})\in M_-(s): x_{3}>\alpha \},$$
and
$${M}^-_+(s):=\{(x_{1},x_2,x_{3})\in M_+(s): x_{3}\leq \alpha \},$$
$${M}^-_-(s):=\{(x_{1},x_2,x_{3})\in M_-(s): x_{3}\leq \alpha \}.$$
Since $M_+(s_0)$ is a graph over $\Pi(0)$, there exists $\alpha$ large enough such that
\begin{equation}\label{eq:graph_outside_a_cylinder}
\text{dist}\big[ \xi\left(M^+_+(s_0)\right), \Pi(0) \big]>0.
\end{equation}
We fix such an $\alpha$. From \eqref{eq:graph_outside_a_cylinder} it follows that there exists $\varepsilon_{0}>0$ such that ${M}^+_+(s_{0}-\varepsilon_{0})$
can be represented as a graph over the plane $\Pi(0)$ and furthermore
\begin{equation}\label{eq:reflection_outside_a_cylinder}
{M}^{+*}_{+}(s_{0}-\varepsilon_{0}) \geq M^+_{-}(s_{0}-\varepsilon_{0}).
\end{equation}
Let  us now investigate the lower part of our surface $M^{-}_{+}(s_0)$. Because $s_{0}\in\mathcal{A}$, we can represent $M^{-}_{+}(s_{0})$ as a graph
over the plane $\Pi(0)$. Note that there is no point in $M^-_{+}(s_{0})$ with normal vector included in the plane $\Pi(0)$ since otherwise
$M^-_{+}(s_{0})$ and its reflection with respect to $\Pi(s_0)$ would have the same tangent plane at that point so by the tangency principle at the boundary
$M$ would have been symmetric to a plane parallel to $\Pi(0)$. But this contradicts the asymptotic
behavior of $M$. Consequently,
\begin{equation}\label{eq:normal_vector_in_Pi_at_s0}
\xi\left(M^-_+(s_0)\right)\cap \Pi(0) =\emptyset.
\end{equation}
{\bf Assertion.} {\it There exists $\varepsilon_{1}\in (0,\varepsilon_{0}]$ such that, for all $t\in [s_{0}-\varepsilon_{1},s_{0}]$,}
\begin{equation}\label{eq:normal_vector_in_Pi_during_a_small_interval}
\xi\left(M^-_+(t)\right)\cap \Pi(0) =\emptyset.
\end{equation}
{\it Proof of the assertion.} Suppose to the contrary that such $\varepsilon_1$ does not exist. This implies that for all $i\in\natural{}$ there exists
$t_i\in [s_{0}-1/i,s_{0}]$ such that
$$\xi\left(M^-_+(t_i)\right)\cap \Pi(0) \neq \emptyset.$$
Then there exists a sequence $\{q_{i}\}_{i\in \mathbb{N}} \subset{M}^-_{+}(t_{i})$
such that $\xi(q_{i})\in \Pi(0)$. Only two situations can occur, namely either the sequence $\{q_{i}\}_{i\in\natural{}}$ is bounded or it is unbounded.
We will show that both cases lead to a contradiction.

If $\{q_{i}\}_{i\in \mathbb{N}}$ is bounded, then it should have a convergent subsequence that we do not relabel for simplicity.
Denote its limit by $q_{\infty}$. Note that $q_{\infty}$ belongs to the closure of $M^-_{+}(s_{0})$. Hence, by the continuity of the Gau{\ss}
map
$$\Pi(0) \supset \mathbb{S}^{1}\ni \xi(q_{i}) \to \xi(q_{\infty}) \in \mathbb{S}^{1} \subset \Pi(0).$$
Then
$$\xi\left(M^-_+(s_0)\right)\cap \Pi(0) \neq \emptyset,$$
which contradicts the relation \eqref{eq:normal_vector_in_Pi_at_s0}.

Let us now examine the case where the sequence $\{q_{i}=(q_{1i},q_{2i},q_{3i})\}_{i\in \mathbb{N}}$ is not bounded. The first coordinate $\{q_{1i}\}_{i\in\natural{}}$
of $\{q_{n}\}_{n\in \mathbb{N}}$ is bounded. The last coordinate $\{q_{3i}\}_{i\in\natural{}}$ of $\{q_{i}\}_{i\in \mathbb{N}}$ is also bounded. Therefore, the second
coordinate $\{q_{2i}\}_{i\in\natural{}}$ of the sequence must be  unbounded. Consider now the sequence $\{M_i=M+(0,-q_{2i},0)\}_{i\in\natural{}}$.
Due to Lemma \ref{pro:limit_process}, we have that after passing to a subsequence, $\{M_{i}\}_{i\in \mathbb{N}}$ converges smoothly to a properly
embedded connected translator $M_{\infty}$ which has the same asymptotic behavior as $M$.
Furthermore, the limiting surface $M_{\infty}$ has the following additional properties:           
\begin{enumerate}[\rm (a)]
\item The surface $( M_{\infty})_{+}(s_{0})$ can be represented as a graph over the plane $\Pi(0)$.
\smallskip
\item The inequality $ (M_{\infty})^*_+(s_{0}) \geq (M_{\infty})_-(s_{0})$ holds true.
\smallskip
\item There exists a point in $M_{\infty}$ in which the Gau{\ss} map belongs to the plane $\Pi(0)$.
\end{enumerate}
Applying the tangency principle at the boundary of $(M_{\infty})_{+}^{*}(s_{0})$ and $(M_{\infty})_{-}(s_{0})$
we deduce that $\Pi(s_{0})$ is a plane of symmetry for $M_{\infty}$, something that contradicts the asymptotic behavior of $M_{\infty}$.
This completes the proof of our assertion.

The relation \eqref{eq:normal_vector_in_Pi_during_a_small_interval} implies that, for every $t\in [s_{0}-\varepsilon_{1},s_{0}]$,
the surface $M^-_{+}(t)$ can be represented as a graph over $\Pi(0)$. Consequently, $M_{+}(t)$ is a graph over $\Pi(0)$ for all
$t\geq s_{0}-\varepsilon_{1}$. Hence the first condition in the definition of the set $\mathcal{A}$ is verified.

\textit{Condition $2$}: Reasoning again as in \cite[Proof of Theorem A]{fra14} and with the help of Lemma 3.1 we can prove the inequality
$M^*_+(s_0-\varepsilon_1)\ge M_-(s_0-\varepsilon_1).$

Therefore, by Conditions 1 and 2, we have that $s_0-\varepsilon\in\mathcal{A}$.
This contradicts the fact that $s_0$ is the infimum of $ \mathcal{A}.$ So, $s_0=0$ and this concludes the proof of STEP 5.

{\bf STEP 6:} Let us explore the asymptotic behavior of our translating soliton $M$ as its $x_2$-coordinate function tends to infinity.
\begin{claim} \label{cl:maxmin}
Consider the profile curve 
$\Gamma=M \cap \Pi(0)$.  If the coordinate function $x_3|_\Gamma$ attains its global extremum on $\Gamma$ (maximum or minimum), then $M$ is a grim reaper cylinder.
\end{claim} 

{\it Proof of the claim.} We will distinguish two cases. The idea is to compare
$M$ with a ``half-grim reaper cylinder" at the level where $x_3$ attains its extremum. 

{\it Case A:} Suppose at first that there exists a point $p\in \Gamma$ (see Fig. \ref{profile}) such that
$$l:=x_3(p)={\max}_{\Gamma} x_3.$$
Observe that
$$\partial M_{+}(0)\subset \{(x_1,x_2,x_3)\in\real{3}:x_3 \le l\}.$$
\begin{figure}[h]
\includegraphics[scale=.07]{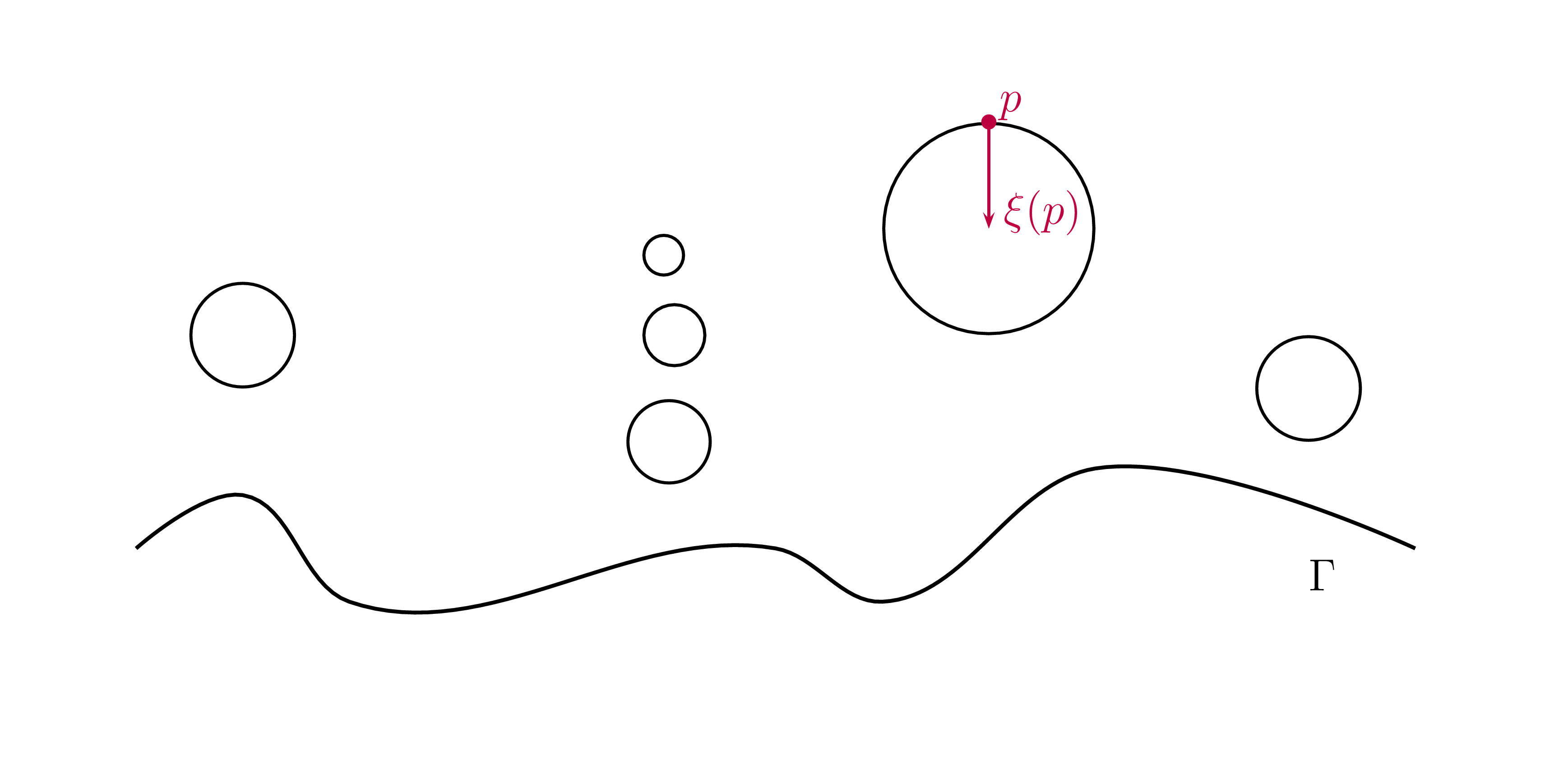}\caption{The profile curve $\Gamma$}\label{profile}
\end{figure}

For a fixed real number $t$ consider the ``half-grim reaper cylinder" (see Fig. \ref{comp-half-grim}) given by
$$\gp^{t,l}_+= \big\{ \big(x_1,x_2,l+\log \cos (x_1-t)\big)\in\real{3}: x_1 \in [t,{\pi}/{2}+t), x_2\in\real{}\big\}.$$
Define now the set
$$\mathcal{Q}:=\big\{ t\in(-\infty,0): \gp^{t,l}_+\cap M_+(0)=\emptyset\big\}$$
Obviously, $\mathcal{Q}$ is a non-empty set. Moreover, if $t \in \mathcal{Q}$
then $(-\infty, t) \subset \mathcal{Q}.$
Let $t_0:=\sup \mathcal{Q}.$
\begin{figure}[h]
\includegraphics[scale=.07]{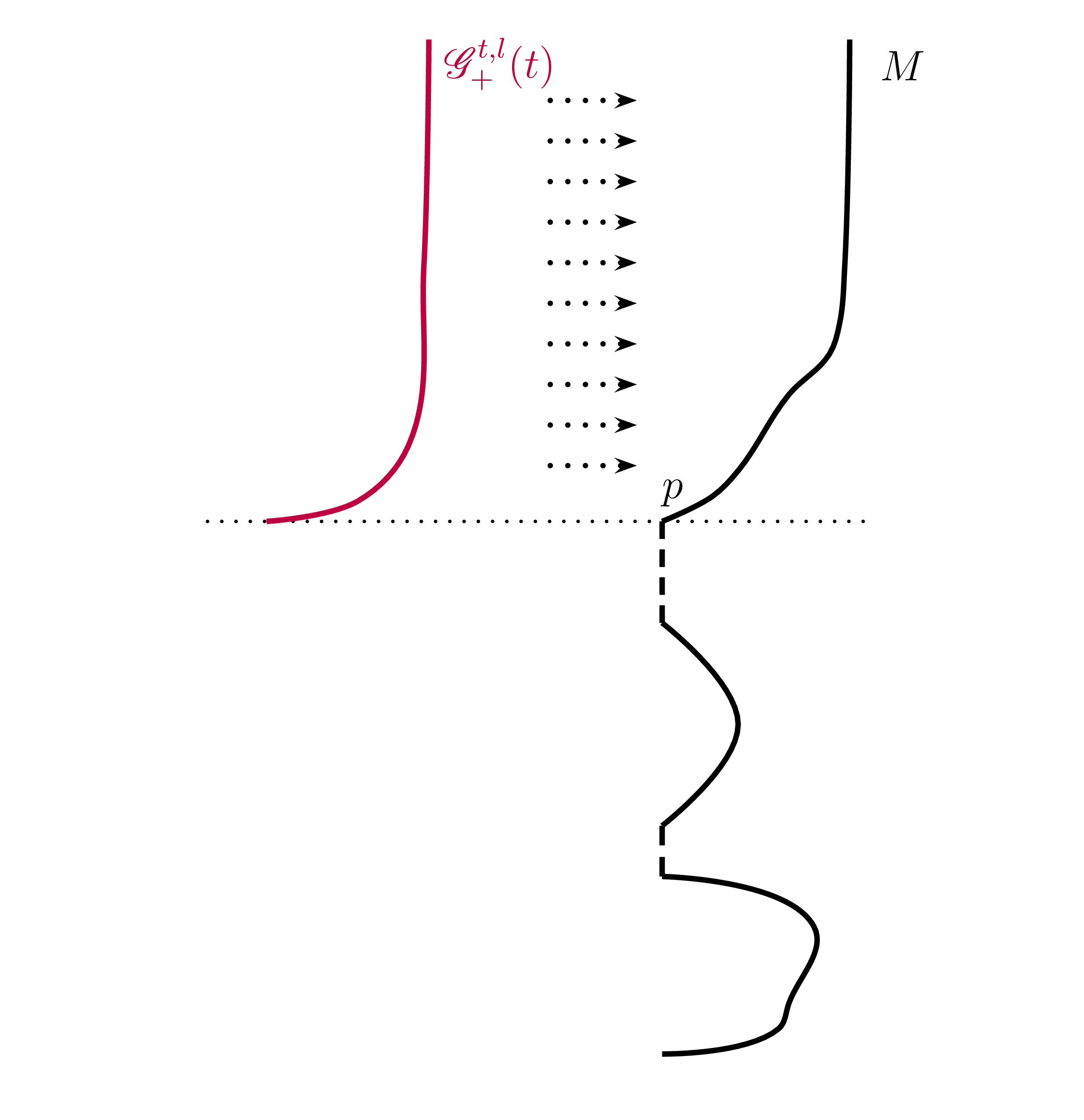}\caption{Comparing with a plane}\label{comp-half-grim}
\end{figure}

We claim that $t_0=0$. Suppose this is not true. If $t_0 \not\in \mathcal{Q}$, then there would be an interior
point of contact (notice that the boundaries of both surfaces do not touch when $t<0$). This implies that
$M=\gp^{t_0,l}$, which contradicts the assumption on the asymptotic behavior of $M$. Let us consider now the
case where $t_0 \in \mathcal{Q}$. In this case there exists a divergent sequence  $\{p_i=(p_{1i},p_{2i},p_{3i})\}_{i \in \n} \subset M_+(0)$
such that
$$\lim_{i\to\infty}\mbox{dist}\big(p_i, \gp_+^{t_0,l}\big)=0.$$
Because the asymptotic behavior of $\gp^{t_0,l}_+$ and $M_+(0)$ is different
and the distance between their boundaries is positive, then one can find
constants $a_0$ and $a_1$ such that $a_0<x_3(p_i)<a_1$, for all $i \in \n$.
So, $\{p_{2 i}\}_{i\in\natural{}}$ tends to infinity.  Now we can apply Lemma \ref{pro:limit_process} in order
to deduce that the limit of the sequence $\{M_i\}_{i\in\natural{}}$, given by
$$M_i:= M-(0,p_{2 i},0),$$
exists and has the same asymptotic behavior as
$M$. Let us call this limit $M_\infty$. But now $M_\infty$ and $\gp^{t_0,l}_+$ have an interior point of contact
and thus they must coincide. This leads again to a contradiction because $M_\infty$ and $\gp^{t_0,l}_+$ do not have 
the same asymptotic  behavior. Hence, $t_0=0$. Consequently, $\gp^{0,l}_+$ and $M_+(0)$ have 
a boundary contact at $p$. Observe that the tangent plane at $p$ of both surfaces is horizontal by STEP $5$, and therefore by the boundary tangency principle
they must coincide.

{\it Case B:}
Suppose now that there exists ${q}\in \Gamma$ such that
$$\mu=x_3(q)={\min}_{\Gamma}x_3.$$
In this case, we compare $M_+(0)$ with the family of ``half-grim reaper cylinders" $\big\{\gp^{t,{\mu}}_+\big\}_{t \ge 0}$ and we
proceed exactly as in the proof of Case A.

\begin{claim} \label{v-graph}
The surface $M$ is a graph over the $x_1x_2$-plane.
\end{claim}

{\it Proof of the claim:} Recall that the profile curve $\Gamma=\Pi(0) \cap M$ lies inside the cylinder $\mathcal{C}$. Let
$$\alpha:= \limsup_{x_2 \to +\infty} \left(x_3|_\Gamma \right)\quad\text{and}
\quad \beta:= \liminf_{x_2 \to -\infty} \left(x_3|_\Gamma \right).$$
Take sequences $\{p_i=(0,p_{2i},p_{3i})\}_{i \in \n}$ and $\{q_i=(0,q_{2i},q_{3i})\}_{i \in \n}$ along the curve $\Gamma$ such that
$$\lim_{i\to\infty} p_{2i}=+\infty,\,\,\lim_{i\to\infty}q_{2i}=-\infty,\,\,\lim_{i\to\infty}p_{3i}=\alpha\,\,\text{and}\,\,\lim_{i\to\infty}q_{3i}=\beta.$$
and define the sequences of translators $\{M^{\alpha}_i\}_{i\in\natural{}}$, $\{M^{\beta}_i\}_{i\in\natural{}}$ given by 
$$M_i^\alpha:= M-(0,p_{2 i},0)\quad\text{and}\quad M_j^\beta:=M-(0,q_{2 j},0).$$
From Lemma \ref{pro:limit_process} we deduce that
$$M^{\alpha}_i\to M^{\alpha}_{\infty}\quad\text{and}\quad M^{\beta}_i\to M^{\beta}_{\infty},$$
where $M^{\alpha}_{\infty}$ and $M^{\beta}_{\infty}$ are connected properly embedded translators with the same asymptotic
behavior as our surface $M$.

Consider the points $(0,0,\alpha) \in M^\alpha_{\infty}$ and $(0,0,\beta) \in M^\beta_{\infty}.$ Taking into account
the way in which we have constructed our limits, we have that 
$$\alpha=\max_{M^\alpha_{\infty} \cap \Pi(0)}x_3\quad\text{and} 
\quad\beta=\min_{M^\beta_{\infty} \cap \Pi(0)} x_3.$$
At this point, we can use Claim \ref{cl:maxmin} to conclude that the limits 
$M_{\infty}^\alpha$ and 
$M_{\infty}^\beta$ are grim reaper cylinders, possibly displayed at different heights. From the definition 
of the limit and the second part of Theorem \ref{thm:compactness_theorem_Brian_White},
it follows that
for large enough values $i\ge i_0$ there exist:
\begin{enumerate}[\rm (a)]
\item
strictly increasing sequences of positive numbers $\{m_{1i}\}_{i\in\natural{}}$, $\{m_{2i}\}_{i\in\natural{}}$, $\{n_{1i}\}_{i\in\natural{}}$ and
$\{n_{2i}\}_{i\in\natural{}}$ satisfying
$$m_{1i}<m_{2i}\quad\text{and}\quad -n_{1i}<-n_{2i},$$
for every $i\ge i_0$,
\smallskip
\item
real smooth functions $\varphi_i:(-\pi/2,\pi/2)\times (m_{1i},m_{2i})\to\real{}$ and $\vartheta_i:(-\pi/2,\pi/2)\times (-n_{1i},-n_{2i})\to\real{}$ satisfying
the conditions
$$|\varphi_i|<1/i,\,\,|\vartheta_i|<1/i,\,\, |D\varphi_i|<1/i\,\,\text{and}\,\,|D\vartheta_i|<1/i,$$
for any $i\ge i_0$,
\end{enumerate}
such that the surfaces
$$R_i:=\big\{(x_1,x_2,x_3)\in M:m_{1i}<x_2<m_{2i}\big\}$$
and
$$L_i:=\big\{(x_1,x_2,x_3)\in M:-n_{1i}<x_2<-n_{2i}\big\}$$
can be represented as graphs over grim reaper cylinders that are generated by the functions $\varphi_i$ and 
$\vartheta_i$, respectively. From the formula \eqref{gaussmap}, by taking larger $i_0$ if necessary,
we deduce that the strips $\{R_i\}_{i\ge i_0}$ and $\{L_i\}_{i\ge i_0}$ are strictly mean convex and so their
outer unit normals are nowhere perpendicular to $\v=(0,0,1)$. Hence each point has a neighborhood
that can be represented as a graph over the $x_1x_2$-plane. Because the strips $R_i$, $L_i$ under 
consideration are smoothly
asymptotic to strips of the corresponding grim reaper cylinders and because for the grim reaper cylinders it holds
$\langle\xi_u,(0,1,0)\rangle=0$, we deduce that the projections of $R_i$, $L_i$ to the
$x_1x_2$-plane are simply connected sets. Therefore, they can be represented
globally as graphs over rectangles of the $x_1x_2$-plane.

Consider now the compact exhaustion $\{\Lambda_i\}_{i\ge i_0}$ (see Fig. 
\ref{Fig-9}) of the surface $M$ 
given by
$$\Lambda_i:=\big\{(x_1,x_2,x_3)\in M:-a_i\le x_2\le b_i,\,\, x_3\le i\big\} $$
where
$$a_i=(n_{1i}+n_{2i})/2\quad\text{and}\quad b_i=(m_{1i}+m_{2i})/2.$$
\begin{figure}[h]
\includegraphics[scale=.08]{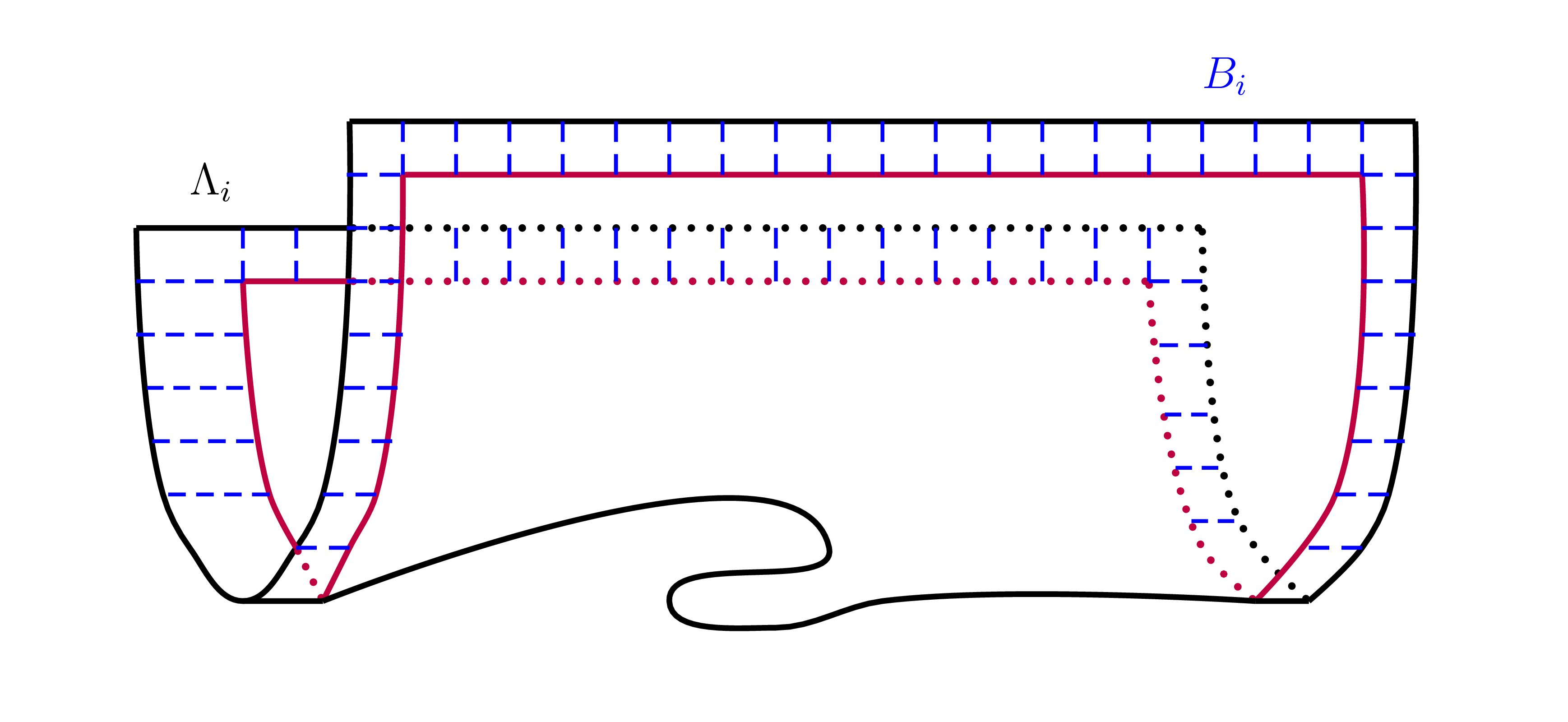}\caption{The exhaustion set $\Lambda_i$}\label{Fig-9}
\end{figure}
The boundary of each $\Lambda_i$ is piecewise smooth and consists of two lateral curves that converge to 
grim reapers and two top curves that converge to two parallel horizontal lines. Observe that in a strip $B_i$ around
$\partial\Lambda_i$ (see again Fig. \ref{Fig-9}) the surface $\Lambda_{i}$ is a graph over the $x_1x_2$-plane. The proof will be concluded 
if we prove that  there exists $i_1\ge i_0$ such that each $\Lambda_i$ is a graph over the $x_1x_2$-plane, for any $i\ge i_1$. Indeed,
at first fix a large height $t_0$ such that $M^+(t_0)$ is a graph over the $x_1x_2$-plane. From Claim \ref{wm} we know that
$$\dist\big(M^-(t_0),\Pi(\pi/2)\big)=\dist\big(\partial M^-(t_0),\Pi(\pi/2)\big)=:\delta.$$
From the asymptotic behavior of $M$ we know that there exists a number $t_1>t_0$ such that
$$\dist\big(M^-(t_1),\Pi(\pi/2)\big)=\dist\big(\partial M^-(t_1),\Pi(\pi/2)\big)=\delta/2.$$
Now fix an integer $i_1>\max\{i_0,t_1\}$, and suppose to the contrary that there is $i\ge i_1$ such that
$\Lambda_i$ is not a graph over the $x_1x_2$-plane. We will derive a contradiction. Let
$$\Lambda_i(s):=\Lambda_i+(0,0,s)$$
be the translation of  $\Lambda_i$ in direction of $\v$. Take a number $s_0$ such that
$$\Lambda_i(s_0)\cap\Lambda_i=\emptyset.$$
Start to move back $\Lambda_i(s_0)$ in the direction of $-\v$. Then there exists
$s_1>0$ where $\Lambda_i(s_1)$ intersects $\Lambda_i$. From the choice of $i_1$ we see that the 
intersection points must be interior points of contact. But then, from the tangency principle,
it follows that $\Lambda_i(s_1)=\Lambda_i$, which is a contradiction. Therefore, for each $i>i_1$ the surface $\Lambda_i$
must be a graph over the $x_1x_2$-plane. Because $\{\Lambda_i\}_{i\in\natural{}}$ is a compact 
exhaustion of $M$ we deduce that $M$ itself must be a graph over the $x_1x_2$-plane. In particular, $\mbox{genus}(M)=0$.

{\bf STEP 7:} From Claim \ref{v-graph} we see that our surface $M$ must be strictly mean convex. Consider now the $x_2$-coordinate of
the Gau{\ss} map, i.e., the smooth function $\xi_2:M\to\real{}$ given by
$\xi_2=\langle\xi,\operatorname{e}_2\rangle,$ where here $\operatorname{e}_2=(0,1,0)$.
By a straightforward computation (see for example the paper \cite[Lemma 2.1]{fra14}) we deduce that $\xi_2$ and $H$ satisfy the following partial differential equations
\begin{equation}\label{pde1}
\Delta\xi_2+\langle\nabla\xi_2,\nabla x_3\rangle+|A|^2\xi_2=0
\end{equation}
and
\begin{equation}\label{pde2}
\Delta H+\langle\nabla H,\nabla x_3\rangle+|A|^2H=0,
\end{equation}
where $|A|^2$ stands for the squared norm of the second fundamental form of $M$. Define now the function $h:=\xi_2H^{-1}$.
Combining the equations \eqref{pde1} and \eqref{pde2} we deduce that $h$ satisfies the following differential equation
\begin{equation}\label{finalpde}
\Delta h+\langle\nabla h,\nabla(x_3+2\log H)\rangle=0.
\end{equation}

\begin{claim}\label{lastclaim}
The surface $M$ is smoothly asymptotic outside a cylinder to the grim reaper cylinder.
\end{claim}
{\it Proof of the claim.} Consider the sequence $\{M_i\}_{i\in\natural{}}$ given by
$M_i:=M+(0,0,-i),$
for any $i\in\natural{}$. One can
readily see that for any compact set $K$ of $\real{3}$, it holds
$${\limsup}_{i\to\infty}\mbox{area}\big\{M_i\cap K\big\}<\infty
\quad\text{and}\quad
{\limsup}_{i\to\infty}\mbox{genus}\big\{M_i\cap K\big\}<\infty.$$
From the compactness theorem of White, the sequence
of surfaces $\{M_i\}_{i\in\natural{}}$ converges smoothly (with respect to the Ilmanen's metric) to the union $\Pi(-\pi/2)\cup\Pi(\pi/2)$. Hence,
due to Lemma \ref{asymptotic-lemma}, the wings of the translator $M$ outside the cylinder must be smoothly asymptotic to the corresponding wings of the
grim reaper cylinder. This completes the proof of the claim.

\begin{claim}\label{cl:h_at_infinity}
The function $h$ tends to zero as we approach infinity of our surface $M.$
\end{claim}

{\it Proof of the claim.} Consider the compact exhaustion $\{\Lambda_i\}_{i>i_1}$ defined in the STEP 6. The boundary of each
$\Lambda_i$ consists of four parts, namely:
\begin{eqnarray*}
\Lambda_{1i}:&=&\big\{(x_1,x_2,x_3)\in M: x_1> 0,\,\,-a_i\le x_2\le b_i,\,\, x_3=i\big\},\\
\Lambda_{2i}:&=&\big\{(x_1,x_2,x_3)\in M: x_1< 0,\,\,-a_i\le x_2\le b_i,\,\, x_3=i\big\},\\
\Lambda_{3i}:&=&\big\{(x_1,x_2,x_3)\in M: x_2=-a_i,\,\, x_3\le i\big\},\\
\Lambda_{4i}:&=&\big\{(x_1,x_2,x_3)\in M: x_2=b_i,\,\, x_3\le i\big\}.
\end{eqnarray*}
Bearing in mind the asymptotic behavior of $M$, we deduce that around each boundary curve line there exists a tubular neighborhood that can be represented
as the graph of a smooth function over a slab of the grim reaper cylinder. If $\varphi$ is such a function then, from the
equations \eqref{gm} and \eqref{gaussmap}, we can represent $h$ in the form
\begin{equation}\label{finalh}
h=-\frac{\varphi_{x_2}}{\cos x_1}\cdot\frac{1+\varphi\cos x_1}{1+\varphi\cos x_1+\varphi_{x_1}\sin x_1}.
\end{equation}
Let us examine at first the behavior of $h$ along $\Lambda_{1i}$. Note
that these curves belong to the wings of $M$ outside the cylinder.
Fix a sufficiently small $\varepsilon>0$. Then, there exists $\delta_2>0$ and
large enough index $i_2$ such that
$$M\cap\big\{(x_1,x_2,x_3)\in\real{3}:x_3\ge i_2\big\} $$
can be written as the graph over the grim reaper cylinder of a smooth function $\varphi$ defined in the domain
$T_{\delta_2}:=\big(\pi/2-\delta_2,\pi/2\big)\times\real{}$ satisfying
$${\sup}_{T_{\delta_2}}|\varphi|<\varepsilon,\quad
{\sup}_{T_{\delta_2}}|D\varphi|<\varepsilon\quad\text{and}\quad
{\sup}_{T_{\delta_2}}|D^2\varphi|<\varepsilon.$$
Because for any fixed $x_2$ we have
$$\lim_{x_1\to\pi/2}\varphi=\lim_{x_1\to\pi/2}|D\varphi|=0,$$
we get
\begin{eqnarray*}
|\varphi_{x_2}(x_1,x_2)|&=&\Big|-\int_{x_1}^{\tfrac{\pi}{2}}\varphi_{x_2x_1}(x_1,x_2)dx_1\Big|\le\big(\pi/2-x_1\big)\left|{\sup}_{T_{\delta_2}}\varphi_{x_1x_2}\right|\\
&\le&\big(\pi/2-x_1\big)\varepsilon.
\end{eqnarray*}
Hence, for any $i\ge i_2$, from equation \eqref{finalh} we see
${\sup}_{\Lambda_{1i}}|h|<\varepsilon.$
Because of the symmetry we immediately get that
${\sup}_{\Lambda_{2i}}|h|<\varepsilon.$
On the other hand, recall that the strips $R_i$ and $L_i$ are getting $C^1$-close to 
the corresponding grim reaper cylinders. Hence, there exists an index $i_3\ge i_2$ such that for $i\ge i_3$ we
can represent
$$R_i\cap\big\{(x_1,x_2,x_3)\in\real{3}:x_3\le i_3\big\} $$
as the graph over a grim reaper cylinder of a smooth function $\varphi_i$ defined in a slab of the form
$G_{\delta_3i}:=(-\pi/2+\delta_3,\pi/2-\delta_3)\times (m_{1i},m_{2i})$, where here $\delta_3$ depends
only on $i_3$, satisfying the properties
$${\sup}_{G_{\delta_3i}}|\varphi_i|<\varepsilon\quad\text{and}\quad
{\sup}_{G_{\delta_3i}}|D\varphi_i|<\varepsilon.$$
Exactly the same estimate can be obtained along the strips $L_i$.
Note that in this case the $x_1$-coordinate is not tending to $\pm\pi/2$ and so
$\cos x_1$ is bounded from below by a positive number.
Going now back to equation \eqref{finalh} we obtain that for $i\ge i_3$ we have
$${\sup}_{\Lambda_{4i}}|h|<\varepsilon\quad\text{and}\quad
{\sup}_{\Lambda_{3i}}|h|<\varepsilon.$$
Therefore $h|_{\partial\Lambda_i}$ becomes arbitrary small as $i$ tends to infinity. This completes the proof of the claim.

From Claim \ref{cl:h_at_infinity}, there exists an interior point where $h$ attains a local maximum or a local minimum. From the strong maximum principle
of Hopf we deduce that $h$ must be identically zero. Consequently, $\xi_2=0$ and thus
$\operatorname{e}_2=(0,1,0)$ is a tangent vector of $M$. Differentiating the equation $h=0$, we deduce that $A(\operatorname{e}_2)=0$.
Thus, $\det A=0$ and so $|A|^2=H^2$. But then, from \cite[Theorem B]{fra14}, we deduce that $M$ should be a grim reaper cylinder.

{\bf Acknowledgments:} {\small  The authors would like to thank Brian White and Antonio Ros for many stimulating and helpful conversations.
Moreover, the authors would also like to thank Oliver Schn\"urer and Leonor Ferrer for plenty of useful discussions.
Finally we would like to thank the referee for the valuable comments on the content of the manuscript and the suggestions for improving the paper.}

\end{document}